\documentclass[11pt]{article}

\usepackage{amsfonts}
\usepackage{amssymb}
\usepackage{amsmath}
\usepackage{amsthm}
\usepackage{graphicx}
\usepackage{float}
\usepackage{bm}
\usepackage{color}
\usepackage{subcaption}
\usepackage{cite}

\newcommand{\bsx}{{\bf x}} 
\newcommand{\bsy}{{\bf y}}

\newtheorem{exmp}{Example}[section]

\theoremstyle{plain}
\newtheorem{theorem}{Theorem}

\theoremstyle{remark}
\newtheorem{remark}{Remark}

\vspace{-2cm}
\title{Direct/iterative hybrid solver for\\ scattering by inhomogeneous
  media} \author{Oscar P. Bruno\footnote{Computing and Mathematical
    Sciences, Caltech, Pasadena, CA 91125, USA, obruno@caltech.edu} \and Ambuj Pandey\footnote{Indian Institute of Science Education and Research Bhopal (IISER Bhopal), ambuj@iiserb.ac.in}}
\date{}

\begin{document}
	\maketitle
\begin{abstract}
  This paper presents a fast high-order method for the solution of
  two-dimensional problems of scattering by penetrable inhomogeneous
  media, with application to high-frequency configurations containing
  (possibly) discontinuous refractivities.  The method relies on a
  hybrid direct/iterative combination of 1)~A differential volumetric
  formulation (which is based on the use of appropriate Chebyshev
  differentiation matrices enacting the Laplace operator) and, 2)~A
  second-kind boundary integral formulation (which, once again,
  utilizes Chebyshev discretization, but, in this case, in the
  boundary-integral context).  The approach enjoys low dispersion and
  high-order accuracy for smooth refractivities, as well as
  second-order accuracy (while maintaining low dispersion) in the
  discontinuous refractivity case.  The solution approach proceeds by
  application of Impedance-to-Impedance (ItI) maps to couple the
  volumetric and boundary discretizations. The volumetric linear
  algebra solutions are obtained by means of a multifrontal solver,
  and the coupling with the boundary integral formulation is achieved
  via an application of the iterative linear-algebra solver GMRES.  In
  particular, the existence and uniqueness theory presented in the
  present paper provides an affirmative answer to an open question
  concerning the existence of a uniquely solvable second-kind
  ItI-based formulation for the overall scattering problem under
  consideration.  Relying on a modestly-demanding scatterer-dependent
  precomputation stage (requiring in practice a computing cost of the
  order of $O(N^{\alpha})$ operations, with $\alpha \approx 1.07$, for
  an $N$-point discretization and for the relevant
    Chebyshev accuracy orders $q$ used), together with fast
  ($O(N)$-cost) single-core runs for each incident field considered,
  the proposed algorithm can effectively solve scattering problems for
  large and complex objects possibly containing discontinuities and
  strong refractivity contrasts.
\end{abstract}

\section{Introduction}
\label{sec:Introduction}
This paper considers the problem of evaluation of wave scattering by
penetrable inhomogeneous media in two dimensions. This is a problem of
fundamental importance in a wide range of applications, including
underwater acoustics, biological and medical imaging, seismology and
geophysics, etc. In all of these applications, it is highly desirable
to utilize efficient and accurate numerical methods which can deal
with arbitrary scattering geometries and (often discontinuous)
refractive index distributions, even in the high-frequency regime. As
is well
known~\cite{gillman2014spectrally,zepeda2016fast,ying2015sparsifying},
this problem presents a number of challenges, as it requires use of
large numbers of discretization points and, for iterative solvers,
increasingly large numbers of iterations as the frequencies and/or
refractive-index values increase. This paper presents a hybrid
iterative/direct linear algebra formulation for this problem, which,
like the approach~\cite{kirsch1994analysis,kirsch1990convergence},
combines a volumetric differential formulation in a bounded region,
and a surface boundary integral equation that provides the coupling to
the complementary unbounded exterior domain.  Unlike the previous
volumetric/boundary
formulation~\cite{kirsch1994analysis,kirsch1990convergence}, which
tackles the volumetric problem via a finite-element discretization,
further, the method proposed here utilizes
(i)~Polynomial approximation patches of accuracy of
  finite order $q$ (with, e.g., $q=10,20,40$); (ii)~A high-order
boundary integral formulation, as well as, both, (iii)~A multifrontal
direct linear solver (the Intel MKL implementation of the multifrontal
solver
Pardiso~\cite{schenk2004solving,bollhofer2019large,schenk2000efficient});
and (iv)~The iterative linear solver GMRES. Leveraging a new version
of the smoothing technique~\cite{hyde2005fast}, finally, the proposed
algorithm yields second-order convergence even for discontinuous
refractivities. As a result of these innovations, the proposed
algorithm can be quite effective: after a modestly-demanding
precomputation stage, requiring a computing cost that grows nearly
linearly with the number $N$ of degrees of freedom used
(Figure~\ref{fig:-time_precomputation} demonstrates a
  growth of the order of $\approx O(N^\alpha)$ with
  $\alpha = 1.07$), and at a cost per GMRES iteration that grows
  essentially linearly with $N$, the proposed method can evaluate,
with a favorable number of iterations, scattering by configurations
including large and complex objects as well as strong refractivity
contrasts and discontinuities---with high accuracy and in fast
single-core runs. A variety of numerical experiments
  have shown (cf. Figure~\ref{fig:-time_complexity} and its caption)
  that, as may be expected in view of the algebraic character of the
  precomputation and iteration stages, for each order $q$ and each
  discretization size $N$, the associated computing times are
  essentially constant asymptotically as the frequency $\kappa$
  grows.
\begin{figure}[h!]
  \centering
  \includegraphics[scale=.5]{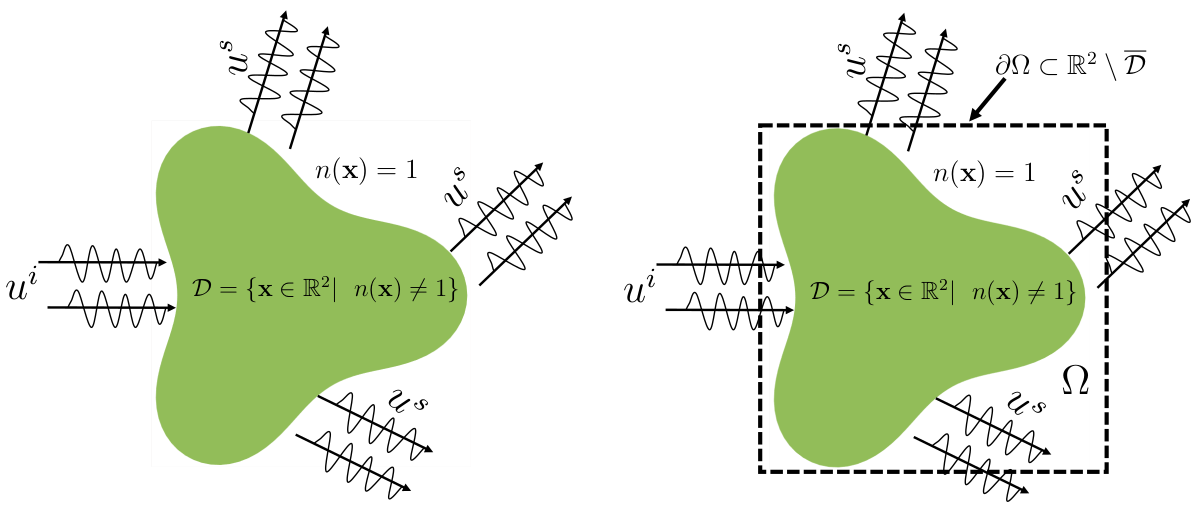}
  \caption{\small  \textbf{Left:}: Scattering by an
      inhomogeneous region
      $\mathcal{D}=\{\bsx \in \mathbb{R}^2 | n(\bsx)\neq 1 \}$.  An
      incident wave $u^{i}$ satisfying \eqref{FHE} impinges upon the
      inhomogeneity $\mathcal{D}$, and thereby the scattered field
      $u^{s}$ (satisfying \eqref{eq:-Sommerfeld})
      results. \textbf{Right:} The unique solution of the problem
      \eqref{FHE}-\eqref{eq:-Sommerfeld}, i.e., the total wave $u$,
      which is equal to the sum of incident wave $u^{i}$ and scattered
      wave $u^{s}$, is computed in a computational domain
      $\overline{\Omega }$ (the square region, enclosed by dotted
      black lines ) containing $\mathcal{D}$, by solving an equivalent
      formulation \eqref{IHE_1}-\eqref{IBIE} in $\overline{\Omega}$.
  }
	\label{scatfig}
\end{figure}

The problem we consider concerns scattering of an incident
time-harmonic acoustic wave $u^{i}$ by a bounded two-dimensional
inhomogeneity
$\mathcal{D}=\{\bsx:n(\bsx)\ne 1\}\subset \mathbb{R}^{2}$, where
$n(\bsx)$ denotes the (possibly discontinuous) index of refraction,
which is assumed to equal unity in the complement
$\mathbb{R}^{2}\setminus \overline {\mathcal{D}}$ of the closure
$\overline {\mathcal{D}}$ of the set $\mathcal{D}$,
as depicted on the left portion of
  Figure~\ref{scatfig}.  Throughout this paper it is assumed that
$u^{i}$ satisfies the free space Helmholtz equation
\begin{equation} \label{FHE} \Delta u^{i}(\bsx) + \kappa^2 u^{i}(\bsx)
  = 0, \ \ \bsx \in \mathbb{R}^2,
\end{equation}
where $\kappa $ is the wave number of the incoming wave $u^{i}$. The
total acoustic field $u$ (which equals the sum $u = u^{i}+u^s $ of the
incident and scattered fields) satisfies the equation
{\cite{colton2013inverse}}
\begin{equation} \label{HE} \Delta u(\bsx) + \kappa^2 n^2(\bsx)
  u(\bsx) = 0, \ \ \bsx \in \mathbb{R}^2,
\end{equation}
and the scattered field $u^s $ satisfies the Sommerfeld radiation
condition %\cite [p-16,p-66]{colton2013inverse}
\begin{equation} \label{eq:-Sommerfeld} \lim_{r \to \infty} \sqrt{r}
  \left( \frac{\partial u^s}{\partial r} - i \kappa u^s \right) = 0,
\end{equation}
where $r = (x_1^2+x_2^2)^{1/2}$ and $i = \sqrt{-1}$ is the imaginary
unit.

The simplest computational approaches to the problem
(\ref{HE})-(\ref{eq:-Sommerfeld}) proceed by replacing the unbounded
propagation region by a bounded computational domain containing the
scatterer $\mathcal{D}$ in its interior (which results in the
introduction of an artificial boundary), and then tackling the
resulting bounded problem by means of finite element or finite
difference discretizations. These approaches yield sparse linear
systems, and, in order to satisfy the radiation condition
(\ref{eq:-Sommerfeld}), they rely on the use of absorbing boundary
conditions. The classical absorbing-boundary
  techniques~\cite{givoli2004abc} and the more recent PML
  approaches~\cite{bermudez2007optimal,cimpeanu2015parameter,johnson2021notes}
  generally require, for accuracy, the use of a relatively large
  distance between scatterers and the absorbing boundary regions, and,
  thus, relatively large computational domains---leading to large
  number of unknowns and correspondingly large linear systems. In
  contrast, as illustrated in Example~4.4, the proposed approach can
  utilize computational boundaries that lie arbitrarily close to the
  scattering surfaces. Further, although square computational
  domains are considered in this paper for definiteness (as depicted
  on the right portion of Figure~\ref{scatfig} and described in detail
  in Section~\ref{prelim}), the proposed algorithm can be generalized
  in a straightforward manner to computational domains consisting of a
  union of disjoint square components covering the region
  $\{n(x)\ne 1\}$ occupied by the scatterer---thus leading, upon use
  of sufficiently small square components, to computational domains
  tightly covering the scattering regions where the refractivity is
  different from the free-space refractivity.  Other
  absorbing-boundary approaches~\cite{hagstrom2004new} allow for the
  use of computational boundaries that lie near the scattering
  boundaries---at the expense of a degree of algorithmic
  complexity. Additionally, the frequently used low-order
  finite-difference (FDM) and finite-element methods (FEM) for the
  problem (\ref{HE}) generally suffer from significant dispersion
  errors \cite{bayliss1985accuracy}, also known as pollution
  error\cite{babuska1997pollution} (a problem which can be alleviated
  or even eliminated~\cite{melenk2011wavenumber} by employing
  high-order finite elements), and they lead to linear systems which
  require large numbers of iterations if treated by means of iterative
  linear algebra solvers~\cite{ernst2012difficult}.

An alternative widely-used computational approach for the problem
(\ref{HE})-(\ref{eq:-Sommerfeld}) relies on the equivalent
\textit{Lippmann-Schwinger} volumetric integral
equation~\cite{colton2013inverse,hyde2005fast}
\begin{equation} \label{eq:-Lippmann} u(\bsx) + \kappa^2 \int
  \limits_{\mathcal{D}}G_{\kappa}(\bsx - \bsy) u(\bsy) m(\bsy)d\bsy =
  u^{i}(\bsx), \hspace{3mm} \hspace{3mm} \bsx \in \mathbb{R}^{2},
\end{equation}
where $ G_{\kappa}(\bsx) = \frac{i}{4}H^{1}_{0}(\kappa|\bsx|) $
denotes the radiating fundamental solution of Helmholtz equation in
free space and $m(\bsx) = 1-n^{2}(\bsx)$ is the contrast
function. This formulation offers several advantages; notably this
approach only requires discretization of the scattering region
$\mathcal{D}$, and the solutions thus obtained automatically satisfy
the Sommerfeld radiation condition~(\ref{eq:-Sommerfeld}).
Additionally, equation (\ref{eq:-Lippmann}) is a Fredholm equation of
the second kind, and, therefore, upon discretization, the condition
number of the resulting linear system remains essentially constant as
the discretization is refined. Unfortunately, however, scattering
solvers based on volumetric integral equation formulations give rise
to certain difficulties, since 1)~The resulting discrete linear
systems, which are dense and non-Hermitian, cannot be effectively
solved by means of classical direct linear-algebra techniques except
for problems that are acoustically small; and 2)~The use of iterative
linear-algebra solvers for such volumetric formulations requires very
large number of iterations for convergence whenever the frequency or
the contrast function $m(\bsx)$ (or both) are large. In recent years,
a number of algorithms, including direct and iterative solvers, have
been proposed for the solution of Lippmann-Schwinger equation, for
instance, see \cite{ ambikasaran2016fast,bruno2004efficient,
  duan2009high, andersson2005fast, vainikko2000fast,bruno2005higher,
  pandey2019improved} and references therein.  The simplest fast
algorithms in this context, which rely on the use of equidistant grids
and FFTs, only provide first order convergence in presence of a
discontinuous index of refraction. For instance, the scheme introduced
in \cite{duan2009high} provides a fast high-order FFT-based method for
smooth refractivities, but it does not yield higher-order accuracy in
presence of discontinuous refractive indices, and it requires large
iteration numbers for high-frequencies or high refractivity contrast.
The algorithm introduced in \cite{bruno2004efficient} exhibits second
order convergence in the presence of discontinuous refractivity, but
this approach does not address problem~2 above: the algorithm requires
large iteration numbers for large frequencies. The recent fast
algorithms~\cite{pandey2019improved,anand2016efficient} provide
convergence-order higher than two via special treatment at
discontinuity boundaries, but they also suffer from large iteration
numbers at high frequency.  Recently preconditioning techniques were
introduced
in~\cite{ying2015sparsifying,liu2018sparsify,zepeda2016fast}, which
were shown to reduce iteration numbers, even at high frequency. No
reports have been provided in either theoretical, graphical or tabular
form, on the numerical accuracy of the solutions provided by these
methods. Further, the effectiveness of these methodologies is highly
dependent on the smoothness of the refractive-index function.  For
example, reference~\cite[Sec. 2.5]{liu2018sparsify} indicates that
``if the [velocity] field has...  discontinuities neither will the
Nystr\"om method be able to give an accurate discretization scheme nor
can the sweeping factorization provide... an accurate approximating
solution. Thus, for our preconditioner to work, we require certain
smoothness from the velocity fields''. 
	
Methods which, like the one proposed in this paper, are based on a
combination of a volumetric differential formulation coupled with a
boundary integral equation for {\em physically-exact} truncation of
the computation domain, have been proposed previously. The first such
contributions were provided
in~\cite{kirsch1994analysis,kirsch1990convergence}, and extensions to
multi-domain iterative solvers in the context of finite-element
discretizations can be found
in~\cite{caudron2020optimized,boubendir2008coupling,bendali2007feti,benamou1997domain}. In
these contexts, the interior volumetric PDE is
generally discretized by means of FEM of low order of
  accuracy, while boundary-element or Nystr\"{o}m discretizations are
  used in the discretization of the boundary integral equation.  As
  mentioned above, the use of low-order FEM methods leads to accuracy
  degradation as the domain sizes grow, in view of the well known
  dispersion errors~\cite{bayliss1985accuracy,babuska1997pollution}
which requires increases in the number of points per wavelength in
order to maintain fixed accuracy as the wavenumber $\kappa$ grows.
High-order methods greatly reduce dispersion and pollution errors, and
they remain advantageous even in presence of discontinuous PDE
coefficients
(Tables~\ref{Table:scat_by_large_smooth_scat},
  \ref{Table:Scat_By_Large_Disc}
  and~\ref{Table:Scat_By_Large_Contrast}). Indeed, the improved
second-order accurate spectral discretization we introduce for
discontinuous-coefficient problems enjoys essentially dispersionless
performance---an important feature that is not obtained
from commonly used low-order finite-difference or
  finite-element methods.

A {\em direct solver} based on spectral discretizations
of fixed order of accuracy, with computational
complexity of order $O(N^{3/2})$, was introduced
in~\cite{gillman2014spectrally}. The method achieves its operation
count by decomposing the domain in a number of spectral square patches
that are organized in a tree structure, with a subsequent aggregation
process, whereby certain ``Impedance-to-Impedance'' (ItI) maps for
individual cells are recursively merged into ItI maps for larger and
larger rectangular groups of cells. Ultimately, when the computational
domain boundary is reached a boundary integral equation is used in
conjunction with the Dirichlet-to-Neumann map (DtN) of the complete
domain to enact the interactions between the bounded scatterer and the
exterior domain. This algorithm can effectively treat high-frequency
problems for which the refractivity is smooth; the illustrations
available in the literature do not include applications for which
refractivity discontinuities exist, but it is expected that the
first-order accuracy would ensue in such cases.

The approach proposed in this paper is a fast hybrid direct/iterative
method which is demonstrated to run at a cost of $O(N^\alpha)$
operations with $\alpha \approx 1.07$, and which, as
  illustrated in Section~\ref{sec:numericalResults}, enjoys a number
  of additional appealing features: the algorithm 1)~Requires a small,
  essentially fixed, numbers of iterations as the refractive index
  (and, thus, the interior wavelength) is increased while keeping the
  exterior wavelength fixed (Table~\ref{Table:Scat_By_Large_Contrast}
  below); 2)~Requires significantly milder increases in iteration
  numbers (see Tables~\ref{Table:scat_by_large_smooth_scat}
  and~\ref{Table:Scat_By_Large_Disc} and Remark~\ref{it_num_rem}
  below) than other iterative
  solvers~\cite{bruno2004efficient,laird2002preconditioned} as the
  exterior frequency increases, in view of its resolution of all
  interior multiple scattering events via a direct solver; 3)~Exhibits
  very low dispersion; and, 4)~Converges with high-order accuracy for
  smooth refractivities, and with second-order accuracy (maintaining
  low dispersion) for discontinuous refractivities, as discretizations
  are refined.  This solver relies on a general-purpose sparse direct
solution technique for the volumetric interior problem that, in
particular, enforces the PDE at spectral cell boundaries by matching
``transmission values'' (that is, the values of the solution and its
normal derivative) at such boundaries; and it incorporates a
second-kind integral formulation in conjunction with an ItI map at the
computational domain boundary (instead of the possibly singular DtN
map used in~\cite{gillman2014spectrally}). The algorithm is completed
by means of the iterative linear-algebra solver GMRES.  In particular,
the existence and uniqueness theory presented in the present paper
provides an affirmative answer to an open-question put forth
in~\cite[Sec. 6]{gillman2014spectrally}, concerning the existence of a
uniquely-solvable second-kind formulation---which involves only ItI
maps, and no DtN maps.
	
As indicated in Section~\ref{sec:numericalAlgorithm}, the proposed
hybrid direct/iterative strategy provides significant advantages over
non-hybrid strategies in which either a fully iterative linear algebra
solver is used, or a generic direct fast sparse solver such
as~\cite{davis2004algorithm} is utilized. Indeed, a fully iterative
solver would necessarily require large numbers of iterations in order
to account for the multiple scattering events that take place at
boundaries of discontinuity of the refractive-index function $n$. As
demonstrated in Section~\ref{sec:numericalResults} (example
\ref{SEHA}), on the other hand, the coupling to the boundary integral
solver destroys the sparsity inherent in the interior spectral matrix,
and can thereby significantly hinder an overall direct solver
strategy.  The proposed hybrid strategy achieves the dual goal of
maintaining a reduced iteration count (since the boundary integral
equation, which requires reduced iteration numbers, is the only
equation that is solved iteratively) while maintaining sparsity.
	
The overall proposed formulation can be used in conjunction with any
adequate direct sparse linear algebra solver for the volumetric
portion of the algorithm. If the specialized Helmholtz direct
linear-algebra solver proposed in~\cite{gillman2014spectrally} were
thus used, the resulting approach would accomplish three goals
mentioned in that reference, namely 1)~Use of an exterior solver based
on the ItI (instead of the Dirichlet-to-Neumann map); 2)~Employment of
an overall formulation that is invertible for all frequencies; and
3)~Use of an iterative strategy for the solution of the integral
equation portion of the method. As indicated above, in this paper we
utilize the Intel MKL implementation of the multifrontal solver
Pardiso~\cite{schenk2004solving,bollhofer2019large,schenk2000efficient},
which has shown to provide excellent performance, at nearly linear
computing cost, to tackle the volumetric portion of the problem.  In
all, the proposed approach provides fast and essentially
dispersionless solutions for high-frequency and/or high-contrast
problems, with high-order accuracy for smooth refractivities, and it
maintains second order accuracy for discontinuous refractive indexes
$n$.
	
This paper is organized as follows.
	% In section (\ref{sec:Preliminaries}) we present mathematical
	% formulation of the problem and a succinct overview of existing
	% numerical method for this problem.
The proposed second-kind integro-differential formulation and the
associated solution-uniqueness proof are presented in
Section~\ref{prelim}. Section \ref{sec:numericalAlgorithm} then
presents a detailed description of the proposed algorithm, and Section
\ref{sec:numericalResults} provides a variety of numerical results
demonstrating the character of the proposed methodology. Concluding
remarks, finally, are presented in Section~\ref{conclusion}.

\section{Uniquely-solvable, second-kind integro-differential hybrid
  formulation} \label{prelim} As discussed in the previous section,
the proposed numerical method is based on a reformulation of the
problem~\eqref{FHE}-\eqref{eq:-Sommerfeld} as a combination of a
differential equation formulation in a volumetric region and a
boundary integral equation formulation on the boundary of
the computational domain.  To describe the method we
consider an open bounded ``computational'' domain
$\Omega\subset \mathbb{R}^{2}$ containing the inhomogeneity:
$\overline{\mathcal{D}}\subset \Omega$. As mentioned in
Section~\ref{sec:Introduction} and depicted on the right portion of
Figure~\ref{scatfig}, throughout this paper the domain $\Omega$ is
taken to equal a square for simplicity, but the algorithm can easily
be generalized to allow for computational domains consisting of a
union of disjoint square components tightly covering the region
$\{n(x)\ne 1\}$. Then the complete
problem~\eqref{FHE}-\eqref{eq:-Sommerfeld} is reformulated in terms of
two main elements: 1)~A Helmholtz equation with variable coefficients
in the volumetric region $\Omega$, and; 2)~A boundary integral
equation on $\partial \Omega$ which couples the solution within
$\Omega$ to the solution in the unbounded domain
$\mathbb{R}^{2}\setminus \overline{\Omega}$. In order to proceed with
this plan the following section first discusses a certain
impedance-to-impedance
operators~\cite{gillman2014spectrally,kirsch1994analysis,kirsch1990convergence}
associated with the Helmholtz problems in the interior and exterior of
$\Omega$.
	
\subsection{Interior and Exterior Impedance-to-Impedance
  operators\label{ItI}}
Let $\Omega\subset\mathbb{R}^2$ denote a bounded open domain with a
Lipschitz boundary $\partial\Omega$. Then, for each non-vanishing real
constant $\beta$, the ``exterior'' impedance operator
$T_{\mathrm{ext}}: H^{-\frac 12}(\partial \Omega)\to H^{-\frac
  12}(\partial \Omega)$ is defined by
\begin{equation}\label{T_ext}
  T_{\mathrm{ext}}[\psi](\bsx) =	u_{\mathrm{ext}}(\bsx)- i\beta \frac{\partial u_{ext}}{\partial \bm{\nu}}(\bsx),
\end{equation}
where $\bm{\nu}$ is the unit outward normal vector at
$\partial \Omega$ and where
$u_\mathrm{ext}\in H^1_\mathrm{loc}\left(\mathbb{R}^2\setminus\Omega
\right)$ is the unique radiating solution of the exterior problem:
\begin{equation}\label{IHEext}
  \begin{cases}
    \Delta u_\mathrm{ext}(\bsx) + \kappa^2 u_\mathrm{ext}(\bsx) = 0, \ \  \mathrm{if} \ \bsx \in \mathbb{R}^2\setminus \overline{ \Omega}, \\%\label{IBCext}
    u_\mathrm{ext}(\bsx)+ i\beta \frac{\partial u_\mathrm{ext}}{\partial \bm{\nu}}(\bsx)=\psi(\bsx), \ \ \mathrm{if} \ \bsx \in\partial \Omega;
  \end{cases}
\end{equation}
see~\cite[Theorem 2.3]{cakoni2001direct} and~\cite[Theorem
6.11]{mclean2000strongly} (cf. ~\cite[Theorem 4.12]{kirsch1992remarks}
and~\cite[Sec. 3.2]{kirsch1994analysis} where corresponding results
for smooth boundaries are provided). The definition of the
``interior'' impedance operator
$T_{\mathrm{int}}: H^{-\frac 12}(\partial \Omega)\to H^{-\frac
  12}(\partial \Omega)$ is analogous:
\begin{equation}\label{int_ItI}
  T_{\mathrm{int}}[\phi](\bsx) =	u_{\mathrm{int}}(\bsx)- i\beta \frac{\partial u_{\mathrm{int}}}{\partial \bm{\nu}}(\bsx),
\end{equation}
where $u_{\mathrm{int}}\in H^1(\Omega)$ is the unique
solution of the problem
\begin{equation}\label{IHE}
  \begin{cases}
    \Delta u_{\mathrm{int}}(\bsx) + \kappa^2 n^2(\bsx) u_{\mathrm{int}}(\bsx) = 0, \ \  \mathrm{for} \ \bsx \in \Omega, \\%\label{IBC}
    u_{\mathrm{int}}(\bsx)+ i\beta \frac{\partial u_{\mathrm{int}}}{\partial \bm{\nu}}(\bsx)=\phi(\bsx), \ \ \mathrm{for} \ \bsx \in\partial \Omega.
  \end{cases}
\end{equation}
	
\subsection{Hybrid formulation}
As is known~\cite[Theorem 2.5]{colton2013inverse}, any radiating
solution $u_\mathrm{ext}$ of the Helmholtz equation over the exterior
domain $\mathbb{R}^{2}\setminus \overline{\Omega}$ may be represented
by means of Green's formula
\begin{equation}\label{GFSF}%exterior representation formula
  u_\mathrm{ext}(\bsx) =\int_{\partial \Omega} \left( \frac{\partial G_{\kappa}(\bm{\bsx-\bsy})}{\partial \bm{\nu(\bsy)}}u_\mathrm{ext}(\bsy) -G_{\kappa}(\bsx-\bsy) \frac{\partial u_\mathrm{ext}}{\partial \bm{\nu}}(\bsy)
  \right) ds(\bsy), \quad \bsx \in \mathbb{R}^{2}\setminus \overline{\Omega}
\end{equation}
which, utilizing the jump relations~\cite[Theorem
3.1]{colton2013inverse} of the single- and double-layer potentials on
$\partial \Omega$ yields the relation
\begin{equation}\label{GFSF_Extended}%exterior representation formula
  u_\mathrm{ext}(\bsx)=\frac{u_\mathrm{ext}(\bsx)}{2}+\int_{\partial \Omega} \left( \frac{\partial G_{\kappa}(\bm{\bsx-\bsy})}{\partial \bm{\nu(\bsy)}}u_\mathrm{ext}(\bsy) -G_{\kappa}(\bsx-\bsy) \frac{\partial u_\mathrm{ext}}{\partial \bm{\nu}}(\bsy)
  \right) ds(\bsy), \quad\mbox{for}\quad \bsx \in \partial \Omega.
\end{equation}  
Similarly, an incident field $u^{i}$ (a function that satisfies
equation~\eqref{FHE} throughout $\mathbb{R}^2$) satisfies
	\begin{equation}\label{GFIF_Extended}%exterior representation
0=\frac{u^{i}(\bsx)}{2}+\int_{\partial \Omega} \left(
\frac{\partial G_{\kappa}(\bm{\bsx-\bsy})}{\partial
\bm{\nu(\bsy)}}u^{i}(\bsy) -G_{\kappa}(\bsx-\bsy) \frac{\partial
u^{i}}{\partial \bm{\nu}}(\bsy) \right) ds(\bsy), \quad\mbox{for}\quad
\bsx \in \partial \Omega.
	\end{equation} In the case $u_\mathrm{ext}$ equals the
scattered field $u^s$ resulting from the incident field $u^i$, we may
combine equations~\eqref{GFSF_Extended} and~\eqref{GFIF_Extended} and
obtain the corresponding relation
\begin{equation}\label{GFTF_Extended}%exterior representation
\frac{u(\bsx)}{2}-\int_{\partial \Omega} \left( \frac{\partial
G_{\kappa}(\bm{\bsx-\bsy})}{\partial \bm{\nu(\bsy)}}u(\bsy)
-G_{\kappa}(\bsx-\bsy) \frac{\partial u}{\partial \bm{\nu}}(\bsy)
\right) ds(\bsy)=u^{i}(\bsx), \quad\mbox{for}\quad \bsx \in \partial
\Omega,
\end{equation} for the total field $u = u^i+u^s$.  Clearly,
defining, for $\bsx\in\partial \Omega$, $\phi(\bsx)=\left( u(\bsx)+
  i\beta \frac{\partial u}{\partial \bm{\nu}}(\bsx)\right) $ and
	\begin{equation}\label{ERF}%exterior representation formula
		\mathcal{A}^\mathrm{int}_\mathrm{ext}[\phi](\bsx) =\int_{\partial \Omega} \left(	\frac{1}{2} \frac{\partial G_{\kappa}(\bm{\bsx-\bsy})}{\partial \bm{\nu(\bsy)}}\left( I + T_{\mathrm{int}}\right)[\phi](\bsy) -\frac{1}{2i\beta}G_{\kappa}(\bsx-\bsy) \left(I - T_{\mathrm{int}}\right)[\phi](\bsy)
		\right) ds(\bsy),
	\end{equation}
	equation~\eqref{GFTF_Extended} may be re-expressed in the form
\begin{equation}\label{GF_re-exp}
		\frac{1}{4}\left(I + T_{\mathrm{int}}\right)[\phi](\bsx)
		-	\mathcal{A}_{\mathrm{ext}}^{\mathrm{int}} [\phi](\bsx)=u^{i}(\bsx).
\end{equation}
	
In particular it is easy to check that, given a solution $u$
of~\eqref{HE}-\eqref{eq:-Sommerfeld} and defining
$\phi = u+ i\beta \frac{\partial u}{\partial \bm{\nu}} $ for
$\bsx \in \partial \Omega$, the pair of functions $(u,\phi)$ is a
solution of the problem
\begin{align}\label{IHE_1}
  &\Delta u(\bsx) + \kappa^2 n^2(\bsx) u(\bsx) = 0, \ \  \mathrm{if} \ \bsx \in \Omega, \\\label{IHE_2}
  &\phi(\bsx) - \left(
    u(\bsx)+ i\beta \frac{\partial u}{\partial \bm{\nu}}(\bsx)\right) = 0, \ \ \mathrm{if} \ \bsx \in\partial \Omega,\\\label{IBIE}
  &\frac{1}{4}\left(I + T_{\mathrm{int}}\right)[\phi](\bsx)
    -	\mathcal{A}_{\mathrm{ext}}^{\mathrm{int}} [\phi](\bsx)=u^{i}(\bsx)\ \ \mathrm{for} \ \ \bsx \in  \partial \Omega.
\end{align}
	%Once the solution $u$ is known in $\overline \Omega$ then solution at any $\bsx \in \mathbb{R}^{2}\setminus \overline{\Omega}$ can be obtained by evaluating integrals in~\eqref{GFSF}. 
As shown in the following section, equation~\eqref{IBIE} (and, thus,
the full problem~\eqref{IHE_1}--\eqref{IBIE}) is uniquely
solvable---and the solution $u$ must therefore coincide with the
restriction to $\overline{\Omega}$ of the solution of the original
inhomogeneous scattering
problem~\eqref{FHE}--\eqref{eq:-Sommerfeld}. Once the solution $u$ of
\eqref{IHE_1}-\eqref{IBIE} is obtained for
$\bsx \in \overline{\Omega}$, the scattered field $u^{s}$ (and hence
the total field $u=u^{i}+u^{s}$) at any point
$\bsx \in \mathrm{R}^{2}\setminus \overline{\Omega}$ can be easily
obtained by utilizing the representation formula~\eqref{GFSF} with
$u_\mathrm{ext} = u^{s}$. In other words, the hybrid
integro-differential problem~\eqref{IHE_1}--\eqref{IBIE} is equivalent
to the original inhomogeneous scattering
problem~\eqref{HE}-\eqref{eq:-Sommerfeld}, as claimed.

\begin{remark} \label{smth-dens} The density function
    $\phi$, which, per Theorem~\ref{uniqueness} below, is the unique
    solution of equation~\eqref{IBIE}, might in principle be expected
    to exhibit some sort of singularity at the corners of the square
    $\partial\Omega$; see
    e.g.~\cite{grisvard2011elliptic,zargaryan1984asymptotic}.
    However, in view of~\eqref{IHE_2}, the solution $\phi$ under
    consideration is actually an infinitely differentiable (and,
    indeed, analytic) function along each one of the sides of the
    square $\partial\Omega$. This follows from the relation
    $\phi = u+ i\beta \frac{\partial u}{\partial \bm{\nu}} $ and the
    fact that the solution $u$ is infinitely smooth (and, in fact,
    analytic) in a certain neighborhood of $\partial\Omega$ within
    which the refractive index $n$ is constantly equal to one.
\end{remark}

	%&u(\bsx) = \mathcal{A}_{\mathrm{text}}^{\mathrm{int}}[\phi](\bsx)\ \ \text{for $\bsx$ outside $\Omega$.}\label{IHE_3}
	
\subsection{Uniqueness }
\begin{theorem} [Uniqueness of solution for the second-kind hybrid
  volume-boundary formulation]\label{uniqueness}
  Let $\phi\in H^{-\frac 12}(\partial\Omega)$ denote a solution of
  equation~\eqref{IBIE} with $u^i=0$. Then $\phi=0$.
\end{theorem}
\begin{proof}
  Letting
  $u_\mathrm{ext}\in H^1_\mathrm{loc}\left(\mathbb{R}^2\setminus\Omega
  \right)$ denote the radiating solution of~\eqref{IHEext}
  corresponding to the impedance data $\psi=\phi$ on
  $\partial \Omega$, the Green relation~\eqref{GFSF_Extended}, with
  integral expressions interpreted as
  in~\cite[Thm. 4.4]{mclean2000strongly}, may be re-expressed in the
  form
  \begin{equation} \label{EGRF} \frac{1}{4}\left(I +
      T_{\mathrm{ext}}\right)[\phi](\bsx) -
    \mathcal{A}_{\mathrm{ext}}^{\mathrm{ext}} [\phi](\bsx)=0 \ \
    \text{for} \ \ \bsx \in \partial \Omega,
  \end{equation} 
  where
  \begin{equation}\label{EA}
    \mathcal{A}_{\mathrm{ext}}^{\mathrm{ext}}[\phi](\bsx) =\int_{\partial \Omega} \left(	\frac{1}{2} \frac{\partial G_{\kappa}(\bm{\bsx-\bsy})}{\partial \bm{\nu(\bsy)}}\left( I + T_{\mathrm{ext}}\right)[\phi](\bsy) -\frac{1}{2i\beta}G_{\kappa}(\bsx-\bsy) \left(I - T_{\mathrm{ext}}\right)[\phi](\bsy)
    \right) ds(\bsy).
  \end{equation}
  Equation~\eqref{IBIE} with $u^{i}=0$ and~\eqref{EGRF} can be recast
  in the forms
  \begin{align} \label{ICFE} \frac{T_{\mathrm{int}}[\phi ](\bsx)}{2}
    -\int_{\partial \Omega} \left ( \frac{\partial
        G_{\kappa}(\bm{\bsx-\bsy}) }{\partial \bm{\nu(\bsy)}} -i\eta
      G_{\kappa}(\bsx-\bsy)\right)T_{\mathrm{int}}[\phi ]
    (\bsy)ds(\bsy)&=f(\bsx), \ \ \bsx \in \partial
    \Omega,\\ \label{ECFE} \frac{T_{\mathrm{ext}}[\phi ](\bsx)}{2}
    -\int_{\partial \Omega} \left (\frac{\partial
        G_{\kappa}(\bm{\bsx-\bsy}) }{\partial \bm{\nu(\bsy)}} -i\eta
      G_{\kappa}(\bsx-\bsy)\right)T_{\mathrm{ext}}[\phi]
    (\bsy)ds(\bsy)&=f(\bsx), \ \ \bsx \in \partial \Omega,
  \end{align}
  where
  \begin{equation} \label{CFRHS} f(\bsx)=-\frac{\phi(\bsx)}{2}
    +\int_{\partial \Omega} \left (\frac{\partial
        G_{\kappa}(\bm{\bsx-\bsy}) }{\partial \bm{\nu(\bsy)}}+i\eta
      G_{\kappa}(\bsx-\bsy)\right)\phi (\bsy)ds(\bsy),
  \end{equation}
  and where $\eta = 1/ \beta$.  Clearly, equations~\eqref{ICFE}
  and~\eqref{ECFE} are identical combined field integral equation of
  second kind, with the same right hand side, for the unknowns
  $T_{\mathrm{ext}}[\phi]$ and $T_{\mathrm{int}}[\phi]$,
  respectively. Since, as is well
  known~\cite[p. 51]{colton2013inverse}, the combined field integral
  equation admits unique solutions, it follows that
  $T_{\mathrm{int}}[\phi] =T_{\mathrm{ext}}[\phi]$ or, equivalently,
  \begin{equation}\label{int-ext}
    u_{\mathrm{int}}(\bsx)- i\beta \frac{\partial
      u_{\mathrm{int}}}{\partial \bm{\nu}}(\bsx) = u_{\mathrm{ext}}(\bsx)-
    i\beta \frac{\partial u_{\mathrm{ext}}}{\partial
      \bm{\nu}}(\bsx) \ \ \text{on} \ \ \partial \Omega,
  \end{equation}  
  where $u_{\mathrm{int}}$ is the solution of~\eqref{IHE}.  But,
  from~\eqref{IHE} and~\eqref{IHEext} we have
  \begin{equation}\label{int-ext_2}
    u_{\mathrm{int}}(\bsx)+i\beta \frac{\partial
      u_{\mathrm{int}}}{\partial \bm{\nu}}(\bsx) =\phi(\bsx) =\psi(\bsx) =
    u_{\mathrm{ext}}(\bsx)+ i\beta \frac{\partial
      u_{\mathrm{ext}}}{\partial \bm{\nu}}(\bsx)\quad \mbox{ on
      $\partial \Omega$},
  \end{equation}  
  and, therefore, using~\eqref{int-ext} it follows that
  \begin{equation}\label{dir-neu-match}
    u_{\mathrm{int}}(\bsx) =u_{\mathrm{ext}}(\bsx) \ \ \text{and} \ \ \frac{\partial u_{\mathrm{int}}}{\partial \bm{\nu}}(\bsx) =\frac{\partial u_{\mathrm{ext}}}{\partial \bm{\nu}}(\bsx) \ \ \text{on} \ \ \partial \Omega.
  \end{equation}  
		
  Let us now define
  \begin{equation}\label{U}
    U_{\phi}(\bsx) =
    \begin{cases}
      u_{\mathrm{int}}(\bsx) & \mbox{for} \ \bsx \in \overline{\Omega} \\
      u_{\mathrm{ext}}(\bsx)& \mbox{for} \ \bsx \in \mathbb{R}^2\setminus
                              \overline{\Omega}.
    \end{cases}
  \end{equation}
  Since $u_{\mathrm{ext}}$ is the radiating solution of~\eqref{IHEext}
  and $u_{\mathrm{int}}$ is the solution of~\eqref{IHE}, on account
  of~\eqref{dir-neu-match} it follows that $U_{\phi}$ is the radiating
  solution of the Helmholtz problem~\eqref{FHE}-\eqref{eq:-Sommerfeld}
  throughout $\mathbb{R}^{2}$ with $u^i=0$. Since this problem admits
  a unique solution in $H^2_\mathrm{loc}(\mathbb{R}^2)$~\cite[Theorem
  8.7]{colton2013inverse} we conclude that $U_{\phi}$ vanishes
  identically. In particular, it follows that $u_{\mathrm{int}}=0$
  throughout $ \overline\Omega$ and, thus, $\phi=0$ in
  $\partial \Omega$ in view of \eqref{int-ext_2}. The proof is now
  complete.
\end{proof}
	
Having established the well posedness of the second-kind hybrid
formulation~\eqref{IHE_1}-\eqref{IBIE} we now present, in the next
section, the proposed numerical algorithm for the solution of this
problem.
	
\section{Numerical algorithm\label{sec:numericalAlgorithm}} \label{NA}
The proposed algorithm relies on the
formulation~\eqref{IHE_1}--\eqref{IBIE} in a computational domain
$\Omega$ which, for definiteness, throughout this paper is taken to
equal the square $\Omega = (-a,a)^2$ with a value of $a$ selected in
such a way that $\overline{\mathcal{D}}\subset \Omega$.  The algorithm
consists of two main components, namely 1)~A spectral volumetric
solver of fixed order of accuracy for the Boundary
Value Problem (BVP)~\eqref{IHE} in the domain $\Omega$ for given
impedance data $\phi \in H^{-1/2}(\partial \Omega)$; and 2)~ A solver
for the boundary integral equation~\eqref{IBIE} on $\partial \Omega$,
which couples the solution within $\Omega$ to the solution in the
exterior of that domain.  In order to achieve second-order convergence
for discontinuous scatterers the algorithm utilizes a filtered
Fourier-smoothing technique outlined in Section~\ref{FS}.  The overall
hybrid approach is completed via an application of the iterative
solver GMRES, as detailed in Section~\ref{Coupling}. As mentioned in
Section~\ref{sec:Introduction}, the overall hybrid method meets the
dual goals of achieving reduced iteration numbers while maintaining
the sparsity of the spectral matrix.

%	{\color{blue}
\subsection{Volumetric boundary-value solver}\label{pde_solver}
This section describes our discretization and direct solution strategy
for the BVP~\eqref{IHE} for given values of the impedance $\phi$ on
$\partial\Omega$.  The presentation includes two subsections, covering
the proposed filtered Fourier smoothing technique that enables
second-order convergence even for discontinuous scatterers
(Section~\ref{FS}), and the Chebyshev-based volumetric discretization
used (Section~\ref{PDES}).

\subsubsection{Filtered Fourier smoothing (FFS) of discontinuous
  refractivities} \label{FS}
	
As is well known, discontinuities in the refractive-index $n(\bsx)$
give rise to severe restrictions on the order of accuracy of the
numerical solutions of the scattering
problem~(\ref{FHE})-(\ref{eq:-Sommerfeld}): in such cases only
first-order accuracy is generally obtained. In the context of the
volumetric Lippmann-Schwinger integral-equation solvers, however,
Reference~\cite{hyde2005fast} shows that full second order convergence
can be reinstated for such problems by means of an application of a
certain Fourier-smoothing
technique~\cite{bruno2005higher,hyde2005fast}. In detail, a quadratic
convergence rate toward the solution for the exact refractivity
$n(\bsx)$ results in that context as the discontinuous contrast
function $m(\bsx)=1-n^2(\bsx)$ is replaced by truncations of its
Fourier series of certain orders, with the additional requirement that
sufficiently accurate values of the Fourier coefficients for the exact
discontinuous function $m(\bsx)$ be used; see Remark~\ref{FC-disc}
below.  Since~\eqref{IHE_1}--\eqref{IBIE} is equivalent to the
corresponding Lippmann-Schwinger problem, the same conclusions hold in
our present spectral context as well. In what follows we present a new
version of the Fourier-smoothing approach, which, incorporating a new
filtering component that eliminates a certain erratic convergence
behavior in the un-filtered approach (see
Table~\ref{Table:scat_by_circularInclusion}), is then applied to the
differential equations considered in this paper. The properties of the
resulting Filtered Fourier Smoothing (FFS) method are demonstrated in
practice via a variety of numerical results in
Section~\ref{sec:numericalResults}.
	
To introduce the method, letting $m= 1-n^2(\bsx)$ we re-express the
Helmholtz equation~\eqref{HE} in the form
\begin{equation}\label{HEM}
  \Delta w(\bsx) + \kappa^2 (1-m(\bsx)) w(\bsx) =0;
\end{equation}
in our context the resulting procedure will be applied to the
problem~\eqref{IHE} to obtain the intermediate solutions
$w=u_{\mathrm{int}}$, and, once convergence has been achieved for the
impedance data $\phi$, to produce the corresponding solution $w=u$
of~\eqref{IHE_1}--\eqref{IHE_2}.

As is well known, the Fourier series of the (possibly discontinuous)
function $m(\bsx)$ converges uniformly to $m(\bsx)$ except on the
discontinuity set, around which it suffers the well known
Gibbs-ringing artifact.  Assuming, for notational simplicity, a square
domain $\Omega$ of side $2a$, the FFS approach proposed in this
section utilizes the order-$F$ {\em filtered} truncated Fourier
expansion
\begin{equation}\label{TFS}
  m^{F}(\bsx)=\sum_{\ell_{1}=-F}^{F}\sum_{\ell_{2}=-F}^{F} c_{\ell_{1},\ell_{2}} e^{\frac{\pi i}{a} (\ell_{1}x_{1}+\ell_{2}x_{2})}
\end{equation}
of the $2a$-biperiodic Fourier series of $m$ in $\Omega$, where
$\bsx =(x_1,x_2)$ and where the filtered Fourier coefficient
$c_{\ell_{1},\ell_{2}} $ are given by
\begin{equation}\label{FC}
  c_{\ell_{1},\ell_{2}} =\left(\frac{1}{4a^2}\int_{-a}^{a}\int_{-a}^{a}m(x_{1},x_{2})e^{-\frac{\pi i}{a} (\ell_{1}x_{1}+\ell_{2}x_{2})}dx_{1}dx_{2}\right)\exp \left(-\alpha\left(\left( \frac{2\ell_1}{F}\right)^{2p}+ \left( \frac{2\ell_2}{F}\right)^{2p}  \right)\right).
\end{equation}
Here, $p$ and $\alpha$ are the parameters in the exponential filter
used; following~\cite{amlani2016fc}, throughout this paper the values
$p = 4$ and $\alpha=16 \log 10$ have been used.

\begin{remark} \label{FC-disc} Note that, for a discontinuous function
  $m$, evaluation of the integral~\eqref{FC} via an FFT would yield
  only first-order accurate coefficients---and would ultimately reduce
  the accuracy the overall solver to first order. A fast
  ($O(F^2\log F)$) algorithm for highly accurate evaluation
  of these coefficients follows from application of the
  (one-dimensional) FC-based integration method presented in
  Appendix~\ref{append} to the integral~\eqref{FC} in the $x_1$ and
  $x_2$ directions.
\end{remark}

	%Now,  Helmholtz equation in~\eqref{IHE_1}
	%
	%in the BVP~\eqref{IHE}, we
	%replace Helmholtz equation~\eqref{IHE_1} by~\eqref{HEF} and the numerical
	%solution of newly introduced formulation yield second order convergent
	%solution to the original problem~\eqref{IHE_1}--\eqref{IBIE}, in spite
	%of the low order approximation of $m(\bsx)$ by $m^{F}(\bsx)$ and
	%associated Gibbs errors. 
	
An additional difficulty associated with the smoothing algorithm still
needs to be tackled since, unlike the algorithm~\cite{hyde2005fast},
our strategy relies on use of non-equispaced (Chebyshev) volumetric
grids, and, therefore, a straightforward evaluation of the Fourier
series of the function $m^{F}(\bsx)$ at the required $N$
discretization points, for which an FFT cannot be directly employed,
generally requires an $O(NF^2)$ computational cost. Since,
generically, $F^2=O(N)$, the overall $O(N^2)$ cost of the
straightforward approach is unacceptably large within our scheme. One
can easily expedite this computation, however, by means of the
FFT-refined trigonometric polynomial interpolation method presented
in~\cite{bruno20071}, which yields high-order accuracy while
maintaining computational efficiency. In our context, once accurate
values of the Fourier coefficients $c_{\ell_{1},\ell_{2}}$ have
somehow been obtained, this interpolation approach can be performed as
a two-step procedure:
\begin{enumerate}
\item Evaluate the Fourier series $m^F(\bsx)$ on a sufficiently fine
  equispaced refinement of the associated $F^2$-point FFT grid. (In
  our examples the fine FFT grid is finer than the original grid by a
  factor of four in each dimension.) This step can be performed by
  means of an FFT on a zero-padded version of the sum~\eqref{TFS}, at
  a cost of $O(F^2\log F)$ operations.
\item In order to evaluate $m^F(\bsx_{0})$ for
  $\bsx_{0}=(x_{0},y_{0})\in \overline{\Omega}$, obtain the value of
  $m^{F}(\bsx)$ at a number $R$ of points neighboring $\bsx_{0}$ in
  the fine grid mentioned in point 1., and interpolate to $\bsx_{0}$
  by means of iterated one dimensional polynomial interpolation; see
  e.g.~\cite{bruno20071}. (In our examples we have used
  fifth order Lagrange polynomial interpolation.)
\end{enumerate}
This procedure yields interpolating polynomials that accurately
reproduce the exact values of the truncated Fourier series at an
$O(N\log N)$ computational cost.
\subsubsection{Volumetric discretization} \label{PDES}
	
This section presents the proposed direct solution strategy for the
numerical solution of the BVP~\eqref{IHE}. As discussed in
Section~\ref{FS}, discontinuities in the refractive index $n(\bsx)$,
if any, are dealt with by utilizing the modified BVP
\begin{align}\label{FSHE_1}
  &\Delta w(\bsx) + \kappa^2 \left(1-m^{F}(\bsx)\right) w(\bsx) = 0, \ \  \text{if} \ \bsx \in \Omega, \\ \label{FSHE_2}
  & 
    w(\bsx)+ i\beta \frac{\partial w}{\partial \bm{\nu}}(\bsx) = \phi(\bsx) \ \ \text{if} \ \bsx \in\partial \Omega
\end{align}
instead of the BVP~\eqref{IHE}---a procedure that, according
to~\cite[Corollary 3.9]{bruno2005higher}
(cf. also~\cite{hyde2005fast}) leads to second-order accurate
approximations to the actual solutions of the original
problem~\eqref{IHE} instead of the first-order convergence that would
otherwise result.  (Note that, interestingly, the proof and
illustrations presented in~\cite[Corollary 3.9]{bruno2005higher}
and~\cite{hyde2005fast} are given in the context of integral
formulations of the problem. But, since the integral-equation and PDE
solutions for the Fourier-smoothed problem coincide, the improved
approximation order carries over, as indicated above and demonstrated
in Section~\ref{sec:numericalResults}, to the present differential
formulation.)
	
For the discussion in the present section we assume that the impedance
data $\phi$ in equation~\eqref{FSHE_2} is known on $\partial
\Omega$. We wish to utilize a general-purpose fast sparse direct
solver, such as, e.g., the multifrontal
algorithm~\cite{davis2004algorithm,bollhofer2019large,alappat2020recursive},
for the solution of our discrete version
of~\eqref{FSHE_1}-\eqref{FSHE_2}. Naturally, the performance of sparse
linear-algebra solvers is highly dependent on the sparsity pattern of
the coefficient matrix of the linear system. In view of this fact, we
seek to approximate all necessary differential operators in such a way
that the resulting linear system is as sparse as possible while
maintaining essentially dispersionless approximations and higher order
accuracy.

To do this we approximate the unknown function $w$ and its derivatives
by means of local Chebyshev representations.  In detail, assuming, for
notational simplicity, a square computational domain $\Omega$, the
proposed BVP solver proceeds by first splitting $\Omega$ into a total
of $P\times P$ mutually disjoint square patches
$\Omega_{i,j},1\leq i,j\leq P$, such that
\[
  \overline{\Omega}=\bigcup_{i,j=1}^{P}\overline{\Omega}_{i,j}.
\]
	
Then, the solution of~~\eqref{FSHE_1}-\eqref{FSHE_2} is obtained by
solving the equivalent set of coupled transmission problems
\begin{align}\label{TP_S}
&\Delta w_{i,j}(\bsx) + \kappa^2 (1-m^{F}(\bsx)) w_{i,j}(\bsx) = 0, \ \ \text{if} \ \ \bsx \in \Omega_{i,j}, \\
\label{TP_T}
&	w_{i,j}(\bsx) =	w_{r,s}(\bsx) \quad  \text{and}\quad   \frac{\partial w_{i,j}}{\partial \bm{\nu}_{i,j}}(\bsx) =	 \frac{\partial w_{r,s}}{\partial \bm{\nu}_{r,s}}(\bsx) \quad
\text{if} \ \ \bsx \in \left(\Gamma_{i,j}\cap  \Gamma_{r,s}\right)\setminus \partial \Omega,	\\
\label{TP_E}
&	w_{i,j}(\bsx)+ i\beta \frac{\partial w_{i,j}}{\partial \bm{\nu}_{i,j}}(\bsx) =	
		\phi(\bsx),
\quad \mbox{if}  \ \ \bsx \in \Gamma_{i,j}\cap \partial \Omega,		
\end{align}
($1\leq i,r\leq P,1\leq j,s\leq P$, ), where
$w_{i,j}=w \big|_{\Omega_{i,j}}$, and where $\bm{\nu}_{i,j}$ denotes
the outward unit normal vector for the domain $\Omega_{i,j}$ on the
boundary $\Gamma_{i,j}=\partial\Omega_{i,j}$. For any
  pair of patches that share a common boundary, the
  conditions~\eqref{TP_T} amount to a manifestation, valid for
  smooth solutions, of the (uniquely solvable) weak formulation of
  equations~\eqref{FSHE_1}-\eqref{FSHE_2} in a multi-patch
  decomposition; see e.g. equations~(1.7) and~(1.8) in
  reference~\cite{kirsch1994analysis}.

To obtain the desired
solutions, for a given positive integer $q$ we discretize the closure
$\overline{\Omega}_{i,j}=[a_{i-1},a_{i}]\times [b_{j-1},b_{j}]$ of the
patch $\Omega_{i,j}$ by means of the two-dimensional tensor product
$ \mathcal{N}_{i,j}=\left\{\bsx_{i,j,k,\ell}\big|\, 0\leq k, \ell \leq
  q\right \}$ Chebyshev mesh given by
\[
  \bsx_{i,j,k,\ell}=\left(\frac{a_{i-1}+a_i}{2}+\frac{a_i -
      a_{i-1}}{2}\cos\left(\frac{\pi
        k}{q}\right),\frac{b_{j-1}+b_j}{2}+\frac{b_j -
      b_{j-1}}{2}\cos\left(\frac{\pi \ell}{q}\right)\right).
\]
Equations for the unknown values of $w_{i,j}(\bsx)$ at the grid points
$\bsx=\bsx_{i,j,k,\ell}$ ($1\leq i,j  \leq P$,
$0\leq k \leq q, 0\leq \ell \leq q$) are obtained by
enforcing discrete versions of equations~\eqref{TP_S}, \eqref{TP_T},
and~\eqref{TP_E}, as appropriate, at the discretization points
$\bsx_{i,j,k,\ell}$ (see Remark~\ref{sparse_corners}), via
approximation of the necessary differential operators
$\partial/\partial {\bm{\nu}}_{i,j}$ and $\Delta$ by Chebyshev
spectral differentiation matrices local to the relevant patch(es)
$\Omega_{i,j}$. These Chebyshev-based approximations of derivatives
remain accurate even for large wavenumbers, and, when used for
discretization of the joint transmission problem
(\ref{TP_S})--(\ref{TP_E}), they give rise to a sparse linear systems
of the form
\begin{equation}\label{LS_SD}
  {\bf A}{\bf w} ={\bf b},
\end{equation}
where the entries of the right hand side vector $\bf b$ associated
with observation points $ \bsx_{i,j,k,\ell} \in \Omega$ equal
zero, and where the entries corresponding to boundary points
$ \bsx_{i,j,k,\ell} \in \partial \Omega $ equal
$ \phi\left(\bsx_{i,j,k,\ell}\right) $.  The unknown vector
${\bf w}$, on the other hand, contains the $N$ unknowns
$w_{i,j,k,\ell}$, one corresponding to each point
$ \bsx_{i,j,k,\ell}$, where
  \begin{equation}\label{N_count}
  N = (q+1)^2P^2;
\end{equation}
note that, in particular, different unknowns are used at single
discretization points that are common to two subdomain
boundaries. Owing to its sparse nature, this linear system is suitable
for treatment by sparse linear solvers such as e.g. the
multifrontal-based direct
solver~\cite{davis2004algorithm,bollhofer2019large,alappat2020recursive}.

\begin{figure}[h!]
\centering
\includegraphics[scale=.5]{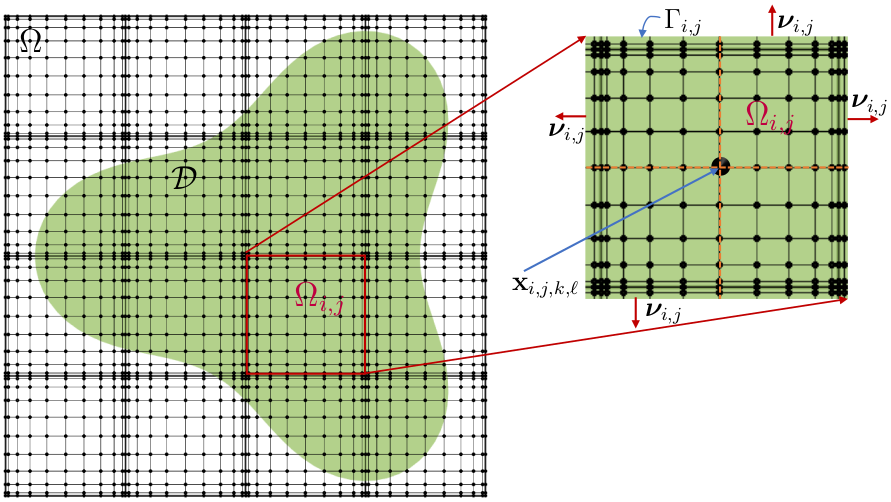}	
%	\subcaption{Each square patecs $\Omega_{i,j}$ is descretized by Chebyshev grids.}
	%	\end{minipage}
\caption{\small Domain-partitioning setup: the computational domain $\Omega$
  containing the inhomogeneity $\mathcal{D}$ is split into
  $P\times P$ Chebyshev patches $\Omega_{i,j}$, $1\leq i,j \leq P$,
   (with $P =4$ in this
  illustration). Derivatives at a given point
  ${\bf x}_{i,j,k,\ell}\in\Omega_{i,j}$ are evaluated as derivatives
  of the Chebyshev expansions obtained from function values along the
  lines passing through ${\bf x}_{i,j,k,\ell}$. The unit normal vector
  $\bm{\nu}_{i,j}$ on the boundary $\Gamma_{i,j}$ of the patch
  $\Omega_{i,j}$ points to the exterior of the patch. }
\label{grid}
\end{figure}
\begin{remark}\label{sparse_corners}
  As illustrated in Figure~\ref{grid}, the only non-zero entries in
  the equation associated with the point $\bsx_{i,j,k,\ell}$
  correspond to discretization points lying on the grid lines that
  pass through $\bsx_{i,j,k,\ell}$.  Two separate
    unknowns are used at each patch-boundary discretization point,
    which are then set to be equal as part of the equation
    system. Similarly, four separate unknowns are used at each patch
    corner point that is not on $\partial\Omega$, and two separate
    unknowns are used at each patch corner point that is on
    $\partial\Omega$.  This strategy is used so as to render each
    patch discretization independent of all other patch
    discretizations at a minimal increase in the number of unknowns.
  A question arises as to which of the two possible enforcements of
  the matching normal derivative conditions in~\eqref{TP_T}, either
  using horizontal or vertical normal derivatives, are used at corner
  points. The indeterminacy is resolved in our algorithm by means of
  the arbitrary but acceptable selection of horizontal normal
  derivatives in~\eqref{TP_T} at all corner points.
\end{remark}

The discrete version of the impedance quantity
$T_{\mathrm{int}}[\phi]$ (equation~\eqref{int_ItI}) which, for a given
$\phi$, is necessary as part of the proposed iterative algorithm for
the solution of~\eqref{TP_S}-\eqref{TP_E} (see
Section~\eqref{Coupling}), can readily be obtained by differentiating
the solution of the linear system~\eqref{LS_SD} on the basis of the
Chebyshev representations introduced in Section~\ref{PDES}.  It is
useful to note that, as indicated in Section~\ref{Coupling}, solutions
of the system~\eqref{LS_SD} with various right-hand sides (one
solution per iteration) are required as part of the proposed iterative
scheme. To obtain the necessary solutions at a reduced computing cost,
in our algorithm the {\bf LU} factorization of the sparse matrix {\bf
  A}, which is obtained by means of the efficient implementation MKL
Pardiso of the multi-frontal sparse
solver~\cite{bollhofer2019large,alappat2020recursive}, is computed
once and stored for repeated use in multiple GMRES iterations, or,
even, for multiple right-hand sides. The computational cost of
assembly of the matrix {\bf A} and evaluation of its {\bf LU}
factorization, whose combination amounts to the most expensive portion
of the overall hybrid volumetric solver, is studied in
Section~\ref{sec:numericalResults}. In particular,
Figure~\ref{fig:-time_complexity} and Table~\ref{Time_MKLPardiso} in
that section demonstrate a computing cost of $O(N^{\alpha})$
operations, with $\alpha \approx 1.07$, for this portion of the
algorithm, with a total number of the order of $O(qN)$ of non-zero
matrix entries.
		
\subsection {High-Order Approximation of the Boundary
  Integral Operators} \label{NS} The proposed algorithm utilizes a
fast, high-order Nystr\"{o}m integration algorithm for the evaluation
of the integral operators in
equation~\eqref{IBIE}. We note without a detailed
  proof that, in view of the smoothness of the solution $\phi$ in a
  neighborhood of $\partial \Omega$ (Remark~\ref{smth-dens}) together
  with stability theory (see e.g.~\cite[Th. 10.2]{kress2013linear}),
  the Chebyshev approximation of the density $\phi$ that we utilize in
  this section gives rise to high-order accuracy in the overall
  algorithm---as illustrated numerically in Table~\ref{smooth_scat}.
In what follows we describe the associated integration scheme and
certain connections with the overall volumetric iterative solver of
which it is a component. Clearly, it is desirable for the underlying
grid in the approximation of~\eqref{IBIE} to be a subset of the
volumetric (piece-wise Chebyshev) interior grid: otherwise an
additional fast and accurate interpolation procedure would be required
for the evaluation of the integral density to the underlying
quadrature points. To avoid such additional difficulties while
preserving maximal accuracy, we use a two-dimensional analog of the
rectangular polar integration scheme recently introduced
in~\cite{bruno2018chebyshev} for the solution of surface scattering
problems in the three dimensions. The resulting procedure is described
in what follows.
	
In a first stage, the entire integration domain $\partial \Omega$ is
covered by a set of non-overlapping boundary patches
$\left\{\gamma_{p}\right\}_{p=1}^{P}$ ($P = 4P$), each one of
which is the image of the interval $[-1,1]$ via a smooth invertible
mapping $\bm{\xi}_{p}$.  Using this covering and the parameterizations
$\bm{\xi}_{p}$, the integral operator that is used as part
of~(\ref{IBIE}) can be decomposed in the form
\[
  \int_{\partial \Omega} \left( \frac{\partial
      G_{\kappa}(\bm{\bsx-\bsy})}{\partial \bm{\nu(\bsy)}} \eta(\bsy)
    -G_{\kappa}(\bsx-\bsy)\zeta(\bsy) \right )ds(\bsy)=\sum_{p=1}^{P}
  I_{p}(\bsx),
\]
	
where
\begin{equation} \label{eq:-parametric_patch_int} I_{p}(\bsx)
  =\int_{-1}^{1} \left( \frac{\partial
      G_{\kappa}\left(\bsx-\bm{\xi}_{p}(t)\right)}{\partial
      \bm{\nu}\left(\bm{\xi}_{p}(t)\right)}
    \eta\left(\bm{\xi}_{p}(t)\right)
    -G_{\kappa}\left(\bsx-\bm{\xi}_{p}(t)\right)\zeta\left(\bm{\xi}_{p}(t)\right)
  \right )\left|\frac{\partial \bm{\xi}_{p}(t) }{\partial t}
  \right|dt.
\end{equation} 
	
An adequate choice of a methodology for the accurate evaluation
of~(\ref{eq:-parametric_patch_int}) depends on the relative position
of the target point $\bsx$ with respect to the integration patch
$\gamma_{p}$.  If the target point $\bsx$ is sufficiently far from
$\gamma_{p}$ then the integrand in~(\ref{eq:-parametric_patch_int}) is
smooth and can be integrated with high-order accuracy by means of any
high-order quadrature rule. On the other hand, if the target point is
either close to or within the integration patch, the integrand is
either singular or near singular, and hence a specialized quadrature
rule must be used for its accurate evaluation.  Thus, depending upon
the distance from the target point to the integration patch, the
overall integration approach relies on three different methods:
	
\textit{Evaluation of non-singular integrals:} For target points
$\bsx$ sufficiently far from the integration patch we use the
Clenshaw-Curtis quadrature which, as is known, enjoys high-order
convergence for smooth integrands~\cite{waldvogel2006fast}, and whose
discretization is taken to coincide with the restriction to
$\gamma_{p}$ of the volumetric discretization
$\cup_{i,j=1}^{P} \ \mathcal{N}_{i,j}$.
	
\textit{Evaluation of singular integrals:} For target points $\bsx$ in
the integration patch, the accurate approximation of
(\ref{eq:-parametric_patch_int}) becomes challenging in view of the
integrand singularity. To deal with this difficulty, we first replace
the density functions $\eta$ and $\zeta$
in~(\ref{eq:-parametric_patch_int}) by their Chebyshev expansions and
we thus obtain
\begin{equation}\label{eq:-cheby_replace}
  I_{p}(\bsx) =\sum_{\ell=0}^{M}c_{\ell}I_{p,\ell}^{1} (\bsx)+ \sum_{\ell=0}^{M}d_{\ell}I_{p,\ell}^{2}(\bsx),
\end{equation}
where
\begin{align}
  \label{SL}
  I_{p,\ell}^{1} (\bsx) &=	\int_{-1}^{1} \frac{\partial G_{\kappa}\left(\bsx-\bm{\xi}_{p}(t)\right)}{\partial \bm{\nu}\left(\bm{\xi}_{p}(t)\right)} T_\ell(t) \left|\frac{\partial \bm{\xi}_{p}
                          (t)}{\partial t}
                          \right|dt,\\
  \label{DL}
  I_{p,\ell}^{2}(\bsx) &=\int_{-1}^{1} G_{\kappa}\left(\bsx-\bm{\xi}_{p}(t)\right)T_\ell(t)\left|\frac{\partial \bm{\xi}_{p}
                         (t)}{\partial t}
                         \right|dt,
\end{align}
and where $T_\ell$ is the Chebyshev polynomial of degree $\ell$. The
Chebyshev coefficients $c _\ell,d_\ell$ can be obtained accurately and
efficiently by means of FFTs. Note that the integrals in
equations~\eqref{SL} and~\eqref{DL} do not depend on the density, and
therefore, may be computed only once and stored for repeated use. In
addition to this, evaluation of these integrals does not require
interpolation, even if refined meshes are used for their evaluation,
as the corresponding integrands are known analytically in the complete
domain of integration. However, evaluation of these integrals present
certain difficulties owing to the weakly singular character of the
integral kernel. To resolve the integrand singularity in
equations~\eqref{SL} and \eqref{DL} we utilize changes of variable
whose Jacobian vanishes along with several of its derivatives at the
singularity point. The idea is not limited to the specific kernel
presently under consideration, and it can be readily incorporated for
a general class of weakly singular kernels. Thus, we present our
discussion in that general context.
	
Letting
\begin{equation} \label{SI} I_\ell(\bsx)=\int_{-1}^{1}
  H_{\kappa}\left(\bm{\xi}_{p}(t_{0})-\bm{\xi}_{p}(t)\right)T_\ell(t)\left|\frac{\partial
      \bm{\xi}_{p} (t)}{\partial t} \right|dt,
\end{equation}
where $\bsx =\bm{\xi}_{p}(t_{0})$ and where
$H_{\kappa}\left(\bm{\xi}_{p}(t_{0})-\bm{\xi}_{p}(t)\right)$ is any
weakly singular kernel, we re-express $I_\ell$ in the form
\begin{equation} \label{SI_Break} I_\ell(\bsx)=\int_{-1}^{t_{0}}
  H_{\kappa}\left(\bm{\xi}_{p}(t_{0})-\bm{\xi}_{p}(t)\right)T_\ell(t)\left|\frac{\partial
      \bm{\xi}_{p}(t) }{\partial t} \right|dt+\int_{t_{0}}^{1}
  H_{\kappa}\left(\bm{\xi}_{p}(t_{0})-\bm{\xi}_{p}(t)\right)T_\ell(t)\left|\frac{\partial
      \bm{\xi}_{p} (t) }{\partial t} \right|dt.
\end{equation}
Both the first and second integrands in~\eqref{SI_Break} are singular
at $t=t_{0}$. To resolve the singularity we use the changes of
variables~\cite{colton2013inverse}
  \[
  t =t_{0}-\frac{1+t_{0}}{\pi}\omega_{k}\left[\frac{\pi}{2}
    (-\tau+1)\right] \quad\text{and}\quad t
  =t_{0}+\frac{1-t_{0}}{\pi}\omega_{k}\left[\frac{\pi}{2} (\tau+1)\right]
\]
in the first and second integrals in~\eqref{SI_Break}, respectively,
which, roughly speaking, distributes half of the
  discretization points near the singular point $t_0$, and the other
  half fairy uniformly throughout the integration
  interval~\cite[p. 84]{colton2013inverse}. Here, for
$0\leq s \leq2\pi$ and for a given integer $k>1$ we have set
\[
  \omega_{k}(s)=2\pi \frac{[v(s)]^{k}}{[v(s)]^{k}+[v(2\pi-s)]^{k}}, \
  \quad \text{where} \quad
  v(s)=\left(\frac{1}{k}-\frac{1}{2}\right)\left( \frac{\pi-s}{\pi}
  \right)^{3}+\frac{1}{k}\left( \frac{s-\pi}{\pi} \right)+\frac{1}{2}.
\]
It is easy to check that the Jacobians of these changes of variables
vanish up to order $k-1$ at the singular point $t=t_{0}$, which
renders smooth integrands that can be integrated with high-order
accuracy by means of the Clenshaw-Curtis quadrature.
	
\textit{Evaluation of near-singular integrals:} This case arises when
the target point $\bsx$ is ``very close'' to, but outside the
integration patch $\gamma_{p}$. In this case, while the integrand
in~\eqref{eq:-parametric_patch_int} is, strictly speaking,
non-singular, its numerical integration poses similar challenges to
the singular case. To effectively treat this issue we project the
target point to the closest point to it on the integration patch and
then follow the same strategy used for singular integration by
treating the projection point as the singular point.
	
\subsection{Overall hybrid solver
  \label{Coupling}} As discussed in the Section~\ref{prelim}, the
proposed method obtains the solution of the scattering
problem~\eqref{HE}-\eqref{eq:-Sommerfeld} by solving the equivalent
formulation~\eqref{IHE_1}--\eqref{IBIE}. If the impedance data $\phi$
in~\eqref{IHE_2} were known on $\partial \Omega$ then the solution of
the scattering problem~\eqref{HE}-\eqref{eq:-Sommerfeld} could be
readily obtained by solving the BVP~\eqref{IHE_1}-\eqref{IHE_2} using
the direct solution algorithm discussed in
Section~\ref{pde_solver}. To obtain $\phi$ on $\partial \Omega$,
equation~\eqref{IBIE} is solved iteratively, where, for each
iteration, the integral operator
$ \mathcal{A}^\mathrm{int}_\mathrm{ext}[\phi]$ (defined
in~\eqref{ERF}) is evaluated via the algorithm discussed in
Section~\ref{NS} in conjunction with the direct solution technique
presented in Section~\ref{pde_solver} for the evaluation of the
interior impedance operator $T_{\mathrm{int}}[\phi]$
(equation~\eqref{int_ItI}). Note that, per the first
  three sentences in Remark~\ref{sparse_corners} and in view
  of~\eqref{N_count}, the sparse $N\times N$ matrix $\mathbf{A}$
  associated with the interior problem (equation~\eqref{LS_SD})
  contains only $2(q+1)N$ non-zero entries.

The main lines of the proposed overall hybrid solver are as follows:
\begin{enumerate}
\item Replace the discontinuous refractivity $n^2(\bsx)$ in
  equation~\eqref{IHE_1} by its filtered Fourier-smoothed version
  $1-m^{F}(\bsx)$ as discussed in Section~\ref{FS}.
		
\item Using either an initial guess (e.g. $\phi = u^i(\bf{x})$) or any
  improved guess for $\phi$ produced by the linear-algebra solver
  GMRES, obtain the solution $u = u_{\mathrm{int}}$ of the
  problem~\eqref{IHE_1}-\eqref{IHE_2}, and then use
  equation~\eqref{int_ItI} to evaluate $T_{\mathrm{int}}[\phi](\bsx)$,
  and, thus, $(I+T_{\mathrm{int}})[\phi]$ and
  $(I-T_{\mathrm{int}})[\phi]$ on $\partial \Omega$.
		
\item Evaluate the left hand side of equation~\eqref{IBIE}, on the
  discretization of $\partial \Omega$, by applying the methods in
  Section~\ref{NS} to integral densities equal to
  $(I+T_{\mathrm{int}})[\phi]$ and $(I-T_{\mathrm{int}})[\phi]$.
		
\item Pass the resulting residual (equal to the difference between the
  left-hand and the right-hand sides in~\eqref{IBIE}) to the GMRES
  algorithm, to obtain a new approximation for the density $\phi$.
\item \label{four} Check for convergence of the density $\phi$ to a
  given prescribed residual tolerance, and iterate by returning to
  step 2 until convergence is achieved.
\item Solve the BVP~\eqref{IHE_1}-\eqref{IHE_2} for the converged
  impedance function $\phi$ obtained per point~\ref{four}. If desired,
  use equation~\eqref{GFSF} with
  $u_{\mathrm{ext}}=\frac 12(I+T_{\mathrm{int}})[\phi]$ and
  $\partial u_{\mathrm{ext}}/\partial \bm{\nu}=\frac{1}{2 i
    \beta}(I-T_{\mathrm{int}})[\phi]$ to produce $u$ in the exterior
  of $\Omega$ and/or, using the Green function
  asymptotics~\cite[Theorem 2.5]{colton2013inverse}, far field values
  for the solution $u$.
\end{enumerate}

\section{Numerical results\label{sec:numericalResults}}
	
This section presents results of numerical tests and examples that
demonstrate the performance of the scattering solvers introduced in
the Section~\ref{NA}, with an emphasis on problems containing
discontinuous refractivities.  All numerical results presented in this
section were produced by means of a C++ implementation of the
algorithms described in Section~\ref{sec:numericalAlgorithm} on a
single core of an Intel i7-4600M processor.  The relative error (in
the near field) reported here was computed according to the expression
\begin{equation*}
  \varepsilon^N_{\infty} = \frac{ \underset{1\leq i \leq N}{\max} \left|u^{\text{ref}}(\bsx_{i})-u^{\text{approx}}(\bsx_{i})\right|}{ \underset{1\leq i \leq N}{\max} \left|u^{\text{ref}}(\bsx_{i}) \right|},
\end{equation*}
where $\{\bsx_{i}\in\Omega:1\leq i\leq N\}$ is a listing of all
volumetric Chebyshev discretization points ${\bf x}_{i,j,k,\ell}$
considered in Section~\ref{PDES}, over all subdomains $\Omega_{i,j}$,
and where $u^{\text{ref}}$ is either a closed form solution, when
available, or a highly accurate numerical solution produced by the
proposed algorithm on a fine discretization.  GMRES tolerances were
prescribed in each case to achieve the desired solution error. Values
of the coupling parameter $\beta$ in the range
$10^{-5}\leq \beta\leq 10^{-3}$ were typically used: as shown in
Figure~\ref{fig:-beta} use of such values of $\beta$ suffices to
eliminate difficulties arising from resonance. (Typically smaller
values of $\beta$ tend to give rise to smaller numbers of iterations,
while slightly larger values of $\beta$ can result in somewhat higher
accuracies; we have found that use of large values of $\beta$, say, in
the range $1\leq \beta \leq 100$, however, can significantly increase
the iteration numbers required to meet a prescribed GMRES tolerance
and/or solution accuracy.)
	% We use the notation $\times n_{1}\times n_{2}$ to specify that
	%
	%$P\times P$ number of non-overlapping patches, each with $N_{1}\times N_{2}$ discretization points  are used for the corresponding numerical solution.
In all of the tabulated results the acronyms ``numIt'' and ``Order''
denote the number of GMRES iterations required to achieve the desired
accuracy and the numerical order of convergence
$\log\left( \varepsilon^{N}_{\infty}/\varepsilon^{2N}_{\infty}
\right)/\log(2)$ respectively.  In accordance with Section~\ref{PDES},
$ P\times P $ denotes the total number of Chebyshev patches used
in the discretization of the computational domain $\Omega$, each one
of which contains $q\times q$ discretization points;
cf. Figure~\ref{grid}.
\begin{figure}[h!]
  \centering
  \includegraphics[width=0.5\linewidth]{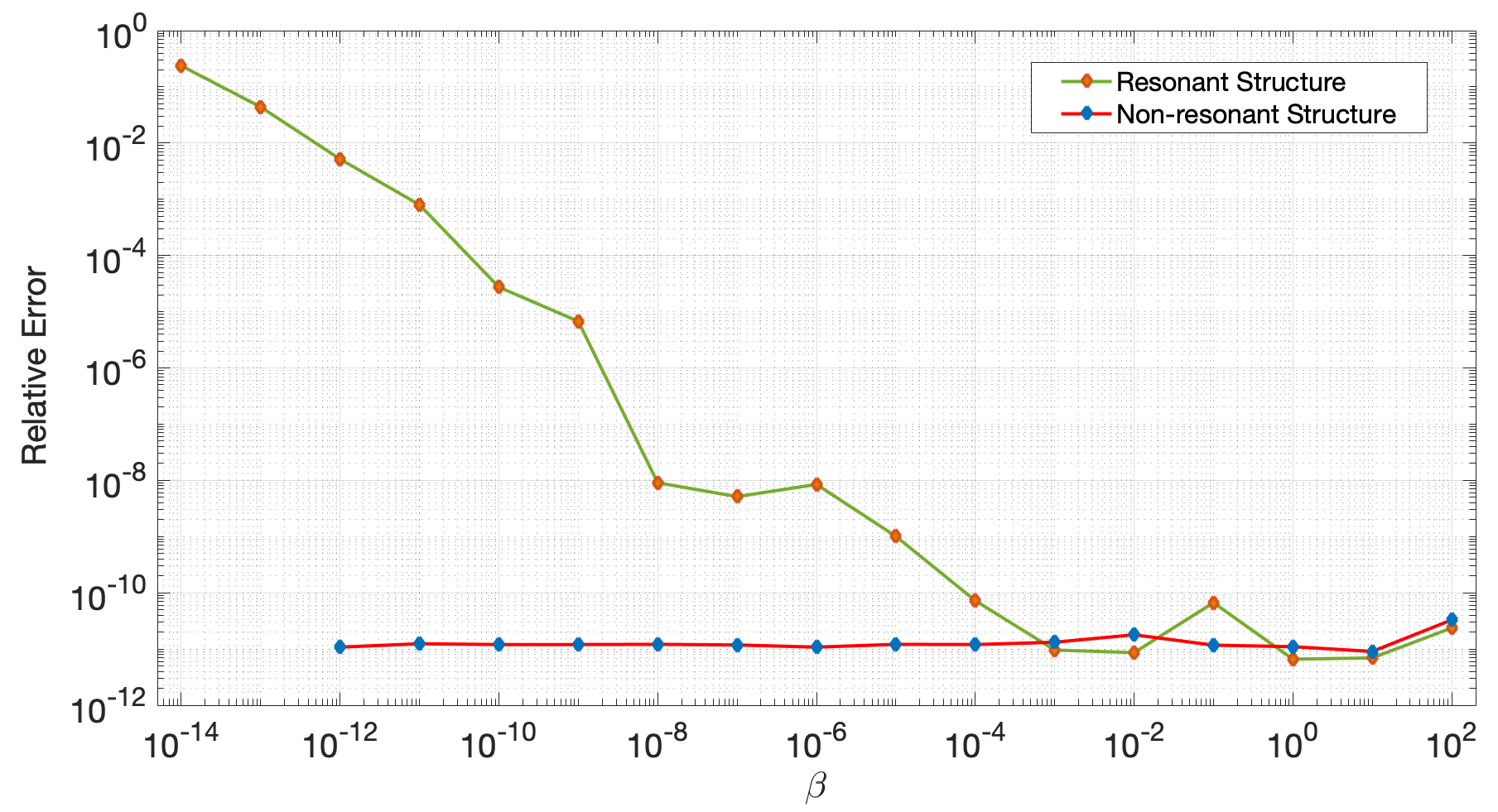}
  \caption{\small Accuracy of the proposed hybrid solver for the problem
    \eqref{IHE_1}--\eqref{IBIE}, as a function of the impedance
    parameter $\beta$ defined in Section~\ref{ItI}, with
    $\kappa=\sqrt{2}\pi$ and $n=1$ in the domain
    $\Omega=(-a,a)\times (-a,a)$ with $a=1$ (``resonant structure'',
    with interior eigenfunction $\sin(\pi x_1)\sin(\pi x_2)$) and
    $a=1.1$ (``non-resonant structure''). In all cases, resonant and
    non-resonant, the GMRES algorithm achieved the $10^{-12}$
    tolerance imposed. As illustrated in the figure, in non-resonant
    cases the error is essentially independent of $\beta$ as
    $\beta\to 0$---since in such cases the ItI map $T_{\mathrm{int}}$
    is well defined for all $\beta$, up to and including $\beta =0$.
    In the resonant case, in contrast, use of a nonzero value of
    $\beta$ is necessary to ensure the map $T_{\mathrm{int}}$ is well
    defined and the algorithms accuracy does not
    deteriorate. \label{fig:-beta}}
\end{figure}

\begin{exmp}(\textit{High-order Convergence for the Boundary Integral
    Operator})
\end{exmp}
This example illustrates the high-order convergence of the singular
integration technique introduced in Section \ref{NS}.
	% In this example, We use the notation $P_{1}\times N_{1}$ to specify that $P$ number of non-overlapping patches, each with $N_{1}$ discretization points  are used for the corresponding numerical solution. 
For our example we let $\kappa =5\pi$,
$v(\bsx)=u^{i}(\bsx) = e^{i\kappa x_1}$ and
$\overline{ \Omega} =\{(x_{1},x_{2})| -1.5\leq x_{1},x_{2}\leq 1.5\}$, and we
evaluate numerically the integral
\begin{equation}\label{eq:-intopt}
  2\int \limits_{\partial \Omega}\left \{G_{\kappa}(\bsx - \bsy) \frac{\partial{v(\bsy)}}{\partial \bm{\nu}(\bsy)}-\frac{G_{\kappa}(\bsx - \bsy)}{\partial \bm{\nu}(\bsy)}v(\bsy)\right \}d\bsy
\end{equation}
for $\bsx\in\partial \Omega$---whose exact value, in view of Green's
theorem, is $ e^{i\kappa x_1}$.  The corresponding results over
successive discretizations are presented in
Table~\ref{Table:PPW_High-order}, clearly demonstrating high-order
accuracy.
	
The proposed integration scheme additionally remains accurate for
large frequencies. To illustrate this, we have computed the integral
(\ref{eq:-intopt}) for various wavenumbers; the corresponding results
are presented in Table~\ref{Table:PPW_disperson} for experiments with
a fixed number of points per
wavelength. Table~\ref{Table:PPW_disperson} shows that, as claimed,
the proposed scheme does not deteriorate as the wavenumber is
increased while keeping a constant number of points per wavelength.
\begin{table}
		%	\begin{minipage}[b]{.48\textwidth}
  \begin{minipage}[t]{0.45\linewidth}\centering
    \centering
    \begin{tabular}{c|c|c|c|c} \hline \hline
      $\kappa$ & $P$ & 	$q$& $\epsilon_{\infty}^{N}$ & Order\\
      \hline
      $5 \pi$& 8 & $6$ & $2.7 \times 10^{-0}$& - \\
      $5 \pi$& $8$ & $12$  &  $5.1 \times 10^{-1}$ &$2.4$ \\
      $5 \pi$& $8$ & $24$ & $ 2.7 \times 10^{-3} $ & $7.5$ \\
      $5 \pi$& $8$  & $48$ &  $3.9\times 10^{-8}$ &$16.1$ \\
      $5 \pi$& $8$ & $96$ & $4.9 \times 10^{-11}$ &$9.6$ \\
      $5 \pi$& $8$ &$192$ &  $1.2 \times 10^{-13}$ &$8.7$ \\
      \hline \hline
    \end{tabular}  
    \caption{\small Convergence study for the singular integration method
      introduced in Section~\ref{NS}.  Numerical errors were obtained
      by comparison against closed-form values of the
      integral~\eqref{eq:-intopt}.}
    \label{Table:PPW_High-order}
  \end{minipage}\hfill
			%	\begin{minipage}[b]{.48\textwidth}
  \begin{minipage}[t]{0.498\linewidth}\centering
    \centering
    \begin{tabular}{c|c|c|c|c}
      \hline
      \hline
      $\kappa$ & $P$&$q$ &  PPW & $\epsilon_{\infty}^{N}$\\ 
      \hline
      $10 \pi$ & $12$ & $30$ & $6 $ &$2.3\times 10^{-5}$ \\ 
      $20 \pi$&24  &30 & 6 &$2.6 \times 10^{-5}$ \\ 
      $40 \pi$ &48 &30  &  6 & $2.7 \times 10^{-5}$ \\ 
      $80 \pi$  & 96  &30 & 6 & $2.9 \times 10^{-5}$\\ 			
      $160\pi$&192 &30 &  6 &$ 3.0 \times 10^{-5}$ \\ 
      $320\pi$&384 &30 &  6 & $3.1 \times 10^{-5}$ \\ 
      \hline	
      \hline
    \end{tabular}  
    \caption{\small Illustration of the proposed
      high-order integration scheme for large
      wavenumbers with a fixed number of points per wavelength.}
    \label{Table:PPW_disperson}
  \end{minipage}
\end{table}

\begin{exmp}
(\textit{Sparsity and Efficiency of the Hybrid Approach})\label{SEHA}
\end{exmp}
As discussed in the introduction, use of the hybrid direct/iterative
strategy, in which the boundary integral equation is treated
iteratively, provides a significant advantage over the corresponding
direct non-hybrid approach, in which a matrix is constructed for the
(complete) coupled volume and boundary discretization. This advantage
arises mainly from sparsity: the matrix ${\bf A}$
(equation~\eqref{LS_SD}) associated with the hybrid approach is
significantly sparser than the corresponding non-hybrid matrix, as the
coupling induced by the boundary integral operator introduces large
numbers of nonzero matrix entries. As a result (and as demonstrated
below in this section) the hybrid approach lends itself much more
effectively to treatment via multifrontal linear-algebra solvers.  To
visualize the source of the sparsity enjoyed by the matrix ${\bf A}$
we note that, in the non-hybrid approach, each boundary entry gives
rise to an equation that links all $4P(q+1)$ boundary unknowns and,
additionally, in view of equations~\eqref{IHE_1} through~\eqref{IBIE},
$(q-1)$ interior unknowns per boundary unknown (as needed to compute
the normal derivative at each boundary point)---so that, in total,
each equation resulting from a boundary point contains $ 4Pq(q+1) $
nonzero entries. This is in contrast to the equations arising from
interior unknowns, each one of which contains merely $2 (q+1)$
non-zero entries. The benefit provided by the hybrid method is that it
decomposes the problem into two parts: a first one that uses a direct
solver for the sparse matrix associated with interior unknowns, and a
second one which treats the boundary unknowns by means of an iterative
procedure.

Table~\ref{H_NH_NNZ}, which displays the total number ``NNZ'' of
non-zero matrix entries contained in the matrices
  treated by means of a direct solver for the hybrid and non-hybrid
methods, demonstrates the sparsity patterns achieved in practice by
the proposed hybrid approach. As illustrated in
Figure~\ref{fig:-time_complexity}, further, such sparsity patterns
translate into fast pre-computation and solution times---which, in
fact, grow nearly linearly with the discretization
size.

\begin{table}[H] \centering
 \begin{tabular}{c|c|c|c|c}
  \hline \hline$P \times P$ & $q \times q$ & \# Bdry. Unknowns & \multicolumn{2}{|c}{ Bdry. unknowns NNZ } \\
  \cline { 4 - 5 } & &$4P (q+1)$  & Non-hybrid & Hybrid \\
  \hline $16 \times 16$ & $10 \times 10$ & $704$ &$5,451,776$ & $7744$ \\
  $32 \times 32$ & $10 \times 10$ & $1408$ &$21,807,104$ & $15488$ \\
  $64 \times 64$ & $10 \times 10$ & $2816$ &$87,228,416$ & $30976$ \\
  $128 \times 128$ & $10 \times 10$ & $5632$ &$348,913,664$ & $61952$ \\
  $256 \times 256$ & $10 \times 10$ & $11264$ & $1,395,654,656$ & $123904$ \\
  \hline
  \hline
 \end{tabular} 
\caption{\small Number NNZ of non-zero matrix entries associated with each boundary unknown for
the non-hybrid and hybrid algorithms, respectively, for various discretization sizes. The greatly enhanced sparsity pattern associated with the hybrid method enables efficient use of multi-frontal linear-algebra
solvers. } \label{H_NH_NNZ}
\end{table}

\begin{table}[H]
	\centering
\begin{tabular}{c|c|c|c|c|c|c|c|c|c}
		\hline \hline$P \times P$ & $N$ & \multicolumn{2}{|c|}{ Pre-comp. (sec.) } & \multicolumn{3}{|c|}{ Per-it. time (sec.) } & \multicolumn{3}{c}{ Memory required (MB) } \\
	 \cline{3-10}  & & $\bf A$/BIE   & {\bf LU}-D & {\bf LU}-inv/It & BIE/It & Tot. &   $\bf A$ storage& {\bf LU}-D & Ratio \\
  \hline $16 \times 16$ & 30,976 & -/0.6 & 0.5 & .04 & .01 & .05 & - & - & - \\
  $32 \times 32$ & 123,904 & .03/3 & 2 & .15 & .06 & 0.21 & - & - & - \\
  $64 \times 64$ & 495,616 & .12/12 & 10& .6 & .3 & 0.9 & 502 & 1,497 & 2.98 \\
  $128 \times 128$ & $1,982,464$ & .46/46 & 46 & 3 & 1 & 4 & 12,99& 8,104 & 6.24 \\
  $175 \times 175$ & $3,705,625$ & .88/87 & 93 & 5 & 2 & 7 & 2,222 & 15,617 & 7.03 \\
  $200 \times 200$ & $4,840,000$ & 1.11/113& 131 & 7 & 3 & 10 & 2,830 & 2,1250 & 7.50 \\
  $256 \times 256$ & $7,929,856$ & 2.00/182 & 247 & 11 & 4 & 15 & 4,484 & 34,624 & 7.72 \\
  $350 \times 350$ & $14,822,500$ & 4/340 & 541 & 22 & 7 & 29 & 8,175& 66,410 & 8.12 \\
  \hline \hline
	\end{tabular}
        \caption{\small Computing times and memory
            required by the various portions of the hybrid algorithm
            on the computational domain $\Omega = [-.5,.5]^2$ with
            $\kappa = 800$. Each one of the $P\times P$ Chebyshev
            patches used was discretized by means of a $q\times q$
            Chebyshev-point discretization with $q=10$. The titles
            used are defined in the text. The BIE precomputation cost
            can be essentially eliminated if an accelerated Green
            function
            method~\cite{bauinger2021interpolated,coifman1993fast} is
            utilized. As indicated in the text, after the
            precomputation stages, small additional memory costs
            suffice to perform even very large numbers of GMRES
            iterations, if needed.}
\label{Time_MKLPardiso} 
\end{table}

Table~\ref{Time_MKLPardiso} and its caption, in turn,
  report computing times and memory required to perform each one of
  the various operations associated with the hybrid method. Thus, in
  particular, for a problem involving nearly 15 million unknowns, the
  single-core precomputation and per-iteration computing times amount
  to $344 + 541\approx 900$ sec. and $22 + 7\approx 30$
  sec. respectively, with a corresponding memory cost of
  $(8,175 + 66,410 + 7,500) $ MB $\approx 82$ GB. (The BIE
  precomputation time and memory cost, the latter one of which is not
  listed in Table~\ref{Time_MKLPardiso}, but which amounts to
  e.g. 7,500 MB for the $\approx 15$ million unknown problem, can be
  essentially eliminated if an accelerated Green function
  method~\cite{bauinger2021interpolated,coifman1993fast} is utilized.)
  An additional (small) memory cost is associated with each GMRES
  iteration: in the $\approx 15$ million unknown test case, for
  example, after an initial integral equation setup memory cost of
  $1,732$ MB, every 100 iterations require a mere $25$ MB of
  additional memory. Thus, in view of
  Table~\ref{Table:Scat_By_Large_Contrast} below, using this
  discretization a solution with an error of the order of $10^{-3}$
  for a domain spanning {\em 350 wavelengths} in diameter containing a
  {\em discontinuous} refractive index can be obtained, on the basis
  of 12 iterations, in a {\em single-core} CPU time of
  $\approx 900 + 12 \cdot 30\ = 1,260$ secs. $= 21$ mins.

(The titles used in Table~\ref{Time_MKLPardiso}
  and~\ref{q10_PC} are defined as follows. The ``Pre-comp'' columns
  list the costs of the various precomputation stages, namely
  ``$\mathbf{A}$'': computing time required to produce the interior
  matrix; ``BIE'': computing time required to evaluate all the
  necessary values of the Green function; and ``LU-D'': computing time
  required to obtain the LU decomposition of $\mathbf{A}$ by means of
  the multifrontal linear-algebra software ``Intel MKL PARDISO''.  The
  ``Per-it'' columns lists costs necessary to perform each iteration,
  namely ``LU-inv/It'': computing time required at each iteration of
  the iterative hybrid algorithm to solve equation~\eqref{IBIE} on the
  basis of the precomputed LU decomposition; and ``BIE/It'': computing
  time required to apply the discrete version of the integral operator
  $T_\mathrm{int}$ in equation~\eqref{IHE_2}, respectively; the column
  ``Tot.'' lists total computing time per iteration.  The ``Memory
  required'' column in Table~\ref{Time_MKLPardiso} lists memory costs,
  including``$\bf A$ storage'': Memory required for storage of the
  matrix $\bf A$ and associated data required by the software``Intel
  MKL PARDISO''; and ``{\bf LU}-D'': Total memory required by the
  solver Pardiso for the precomputation of the {\bf LU} decomposition
  of the matrix $\bf A$; the column ``Ratio'' lists the ratio of the
  memory requirements in the two previous columns.)

\textcolor{blue}{
  \begin{figure}[h!]
	\begin{minipage}[b]{0.45\linewidth}
          \centering
          \includegraphics[width=\linewidth]{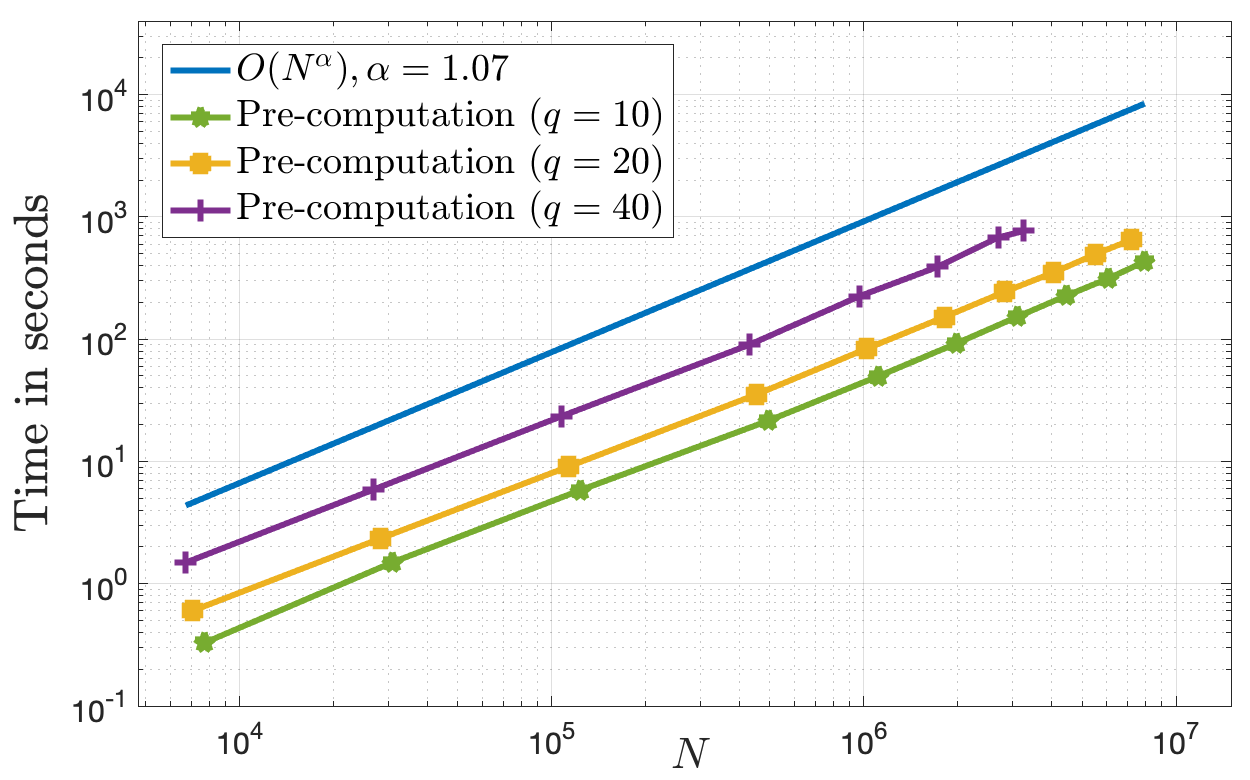}
%          \subcaption{Precomputation costs for $q=10$, $q=20$ and
%            $q=40$, as a function of $N$, vs. a curve
%            $O(N^{\alpha}$) with $\alpha = 1.07$.
%			%. The slope of the lines(using Matlab command ``polyfit") for $q=10,20,40$ are $1.06,1.03,1.03$ respectively.
%		 }
\subcaption{Pre-computation costs for $q=10$, $q=20$, and $q=40$.}
	\label{fig:-time_precomputation} 
	\end{minipage}
	\hspace{0.1cm} 
	\begin{minipage}[b]{0.47\linewidth}
          \centering
          \includegraphics[width=\linewidth]{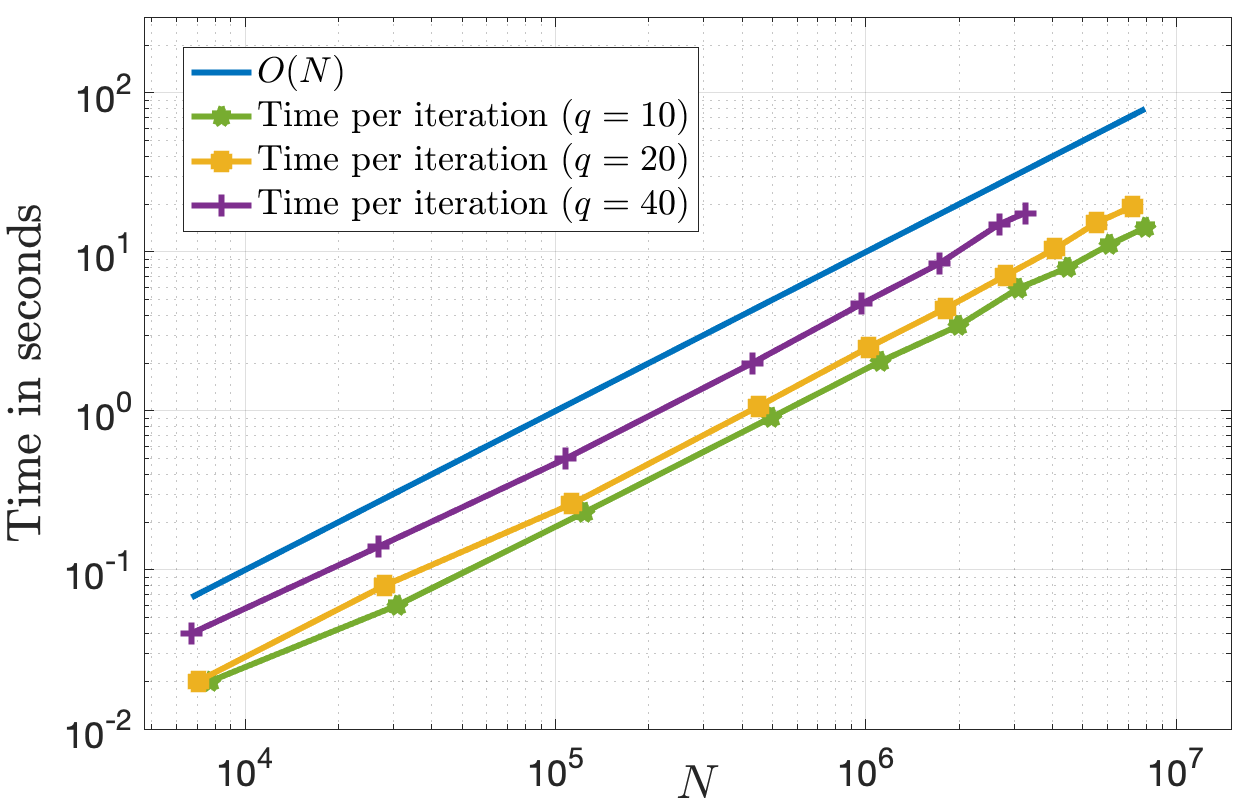}
%          \subcaption{Computing time per iteration versus a curve
%            $ O(N)$.}
\subcaption{Time per GMRES iteration for $q=10,20,$ and $q=40$.}
		\label{fig:-time_per_iteration}
	\end{minipage}
	\caption{\small \textbf{Left:}
            Pre-computation time in seconds (required to build the
            matrix $\bf A$, to obtain its $\bf LU$ decomposition and
            to produce the BIE Green-function precomputation), for
            $q=10$, $q=20$ and $q=40$, as a function of $N$, with
            $\kappa = 800$ and for all discretizations allowable
            within the available memory, vs. a curve $O(N^{\alpha}$)
            with $\alpha = 1.07$. \textbf{Right:} Per-iteration time
            (in sec.)  required by the hybrid algorithm with $q=10$,
            $q=20$ and $q=40$, as a function of $N$, vs. a line
            $O(N)$). As illustrated in Table~\ref{q10_PC}, a variety
            of numerical experiments have shown that these computing
            times are essentially constant asymptotically as $\kappa$
            grows.}
	\label{fig:-time_complexity}
      \end{figure}
    }

    \begin{table}[H]
	\begin{center}
	\begin{tabular}{c|c|c|c|c|c|c}
	\hline \hline
	$ P \times P$ & $q \times q$ & $\kappa$ & \multicolumn{3}{c|}{Pre-comp. (in sec.)} & Per it. time   \\
	\cline{4-6}& & & $\bf A$/BIE  & $\bf LU$-D  &  Tot. & (sec.)  \\
	\hline
%	\hline
%	160 & 10 & 100 & 69.82& 79.23  & 149.05 & 6.04  \\
%	192 & 10 & 100 & 101.01 & 122.22 & 223.23 & 8.43  \\
%	224 & 10 & 100 & 135.36  & 171.06 & 306.42  & 11.88  \\
%	
	
	$256\times 256$ & $10 \times 10$  & 10 & 2.00/134.74 & 246.52 & 383.26 & 15.20  \\
	$256\times 256$ & $10 \times 10$  & 100 & 2.00/174.89 & 246.74 & 423.63 & 15.28  \\
	$256\times 256 $ & $10 \times 10$  & 200 &  2.00/181.32 & 246.56 & 429.88 & 15.42  \\
	$256\times 256$ & $10 \times 10$ & 400 &  2.00/180.03   & 246.18 & 428.21 &  15.49 \\
	$256\times 256 $ & $10 \times 10$  & 800 &   2.00/181.89  & 247.74 & 431.63 & 15.49 \\
	\hline
	\hline
		
        \end{tabular}\caption{\small Pre-computation time for $q=10$ for various values of $\kappa$ with fixed $N = P^2\times (q+1)^2 = 7,929,856$. } \label{q10_PC}
\end{center}
\end{table}

%\begin{figure}[h!]
%\begin{minipage}[b]{0.5\linewidth}
%  \centering
%  \includegraphics[width=\linewidth]{Pre-computation_time.png}
%  \subcaption{Precomputation cost (dotted line) and curve
%    $O(N^{\alpha})$, with $\alpha = 1.07$.}
%\label{fig:-time_precomputation} \end{minipage}
%\hspace{0.1cm} 
%\begin{minipage}[b]{0.47\linewidth}
%\centering
%\includegraphics[width=\linewidth]{time_per_iter.png} 
%\subcaption{
%$O(N)$ cost for per GMRES iteration.} 
%\label{fig:-time_per_iteration}
%\end{minipage}
%\caption{\small  Required time, in seconds, for the pre-computation and iteration stages of the hybrid algorithm. } 
%\label{fig:-time_complexity}
%\end{figure}

\begin{exmp}
	(\textit{Scattering by a Smooth Gaussian Bump}) 
\end{exmp}
This example demonstrates the high-order convergence enjoyed by the
proposed algorithm when applied to smooth contrast functions
$m(\bsx)$. In detail, we consider the total field $u$ that arises
under plane wave excitation incident from the positive $x_1$ axis, for
the contrast function given by the smooth Gaussian bump
$m(\bsx) =-1.5e^{-60|\bsx|^2}$. Numerical results for the domain
$\Omega=(-0.5,0.5)^2$ and $\kappa = 20\pi$, at various discretization
levels, are displayed in Table \ref{smooth_scat}---clearly
demonstrating the high-order convergence of the proposed algorithm for
smooth scattering media.  The extremely low dispersion provided by the
proposed algorithm is demonstrated in
Table~\ref{Table:scat_by_large_smooth_scat} (see also
  Tables~\ref{Table:Scat_By_Large_Disc}
  and~\ref{Table:Scat_By_Large_Contrast})---which shows that, for the
same domain $\Omega$, the accuracy is maintained while keeping the
number of points per wavelength fixed---even for large
frequencies.% In

\begin{table}[h!]  \centering 
\begin{tabular}{c|c|c|c|c|c} 
\hline
\hline
$\kappa$ & $P\times P$ & 	$q\times q$& $\epsilon_{\infty}^{N}$ & Order & $\#$ Iter\\
\hline
$20 \pi$& $2 \times 2$ & $10 \times 10$ & $1.39 \times 10^{-0}$& - & 28\\
$20 \pi$& $4 \times 4$& $10 \times 10$  &  $5.81 \times 10^{-1}$ & 1.22 & 28 \\
$20 \pi$& $8 \times 8$ & $10 \times 10$ &  $9.50 \times 10^{-3}$ & 5.93 & 28 \\
$20 \pi$& $16 \times 16$  & $10 \times 10$  &  $2.11 \times 10^{-5}$ &8.81 & 28 \\
$20 \pi$& $32 \times 32$ & $10 \times 10$ & $1.08 \times 10^{-7}$ & 7.61 & 28 \\
$20 \pi$& $64 \times 64$ & $10 \times 10$ & $1.50 \times 10^{-10}$ &9.49& 28 \\

\hline
\hline \end{tabular}
\caption{\small Convergence
study for the smooth Gaussian bump test case. For these experiments the GMRES residual tolerance and the coupling
parameter $\beta$ were set to $10^{-12}$ and $10^{-5}$, respectively.}  
\label{smooth_scat}
\end{table}

  \begin{remark}\label{it_num_rem}
    It is important to note the relatively mild (roughly linear)
    increases in iteration numbers demonstrated in
    Tables~\ref{Table:scat_by_large_smooth_scat}
    and~\ref{Table:Scat_By_Large_Disc} as the incident frequency grows
    (cf. references~\cite{bruno2004efficient,laird2002preconditioned}). The
    observed linear growth is purely associated with the spectral
    character of the boundary integral operators used, and it results
    as the algorithm bypasses, by means of its interior direct solver
    component, the iterative resolution of all interior
    multiple-scattering events that would otherwise require
    significantly larger iteration numbers.
\end{remark}

 \begin{table}[H] 
 	\begin{center}
 		 \begin{tabular}{c|c|c|c|c|c|c|c}
\hline
\hline
$\kappa$& $P\times P$ &$q\times q$ & $N/\Gamma_{N}$  &$\epsilon_{\infty}^{N}$ &$\#$ Iter.& \multicolumn{2}{c}{Time (sec.)} \\
\cline{7-8} & & & & & &pre-comp& per. It.\\
\hline
% $20 $ & $4\times 4$ & $12 \times 12$ & $2.4 \times 10^{-4}$& $4$ & $0.3$ &$-$ \\
\hline
$50 $ & $10 \times 10$ & $10 \times 10$ &  $14641/484$ &$2.75 \times 10^{-5}$ & $14$ & $.63$ &$0.03$\\
\hline
$100$ & $22\times 22$ & $10 \times 10$  & $58564/968$ &$6.66\times 10^{-5}$ & $35$ & $2.35$ &$0.1$\\
\hline
$200$ & $44\times 44$ & $10 \times 10$  & $234256/1936$ & $1.09\times 10^{-4}$ & $76$ & $9.68$ &$0.41$\\
\hline
$400$ & $88\times 88$ & $10 \times 10$  & $937024/3872$ &$2.12\times 10^{-4}$ & $162$ & $43$ &$1.62$\\
\hline
$800$ & $176\times 176$ & $10 \times 10$  & $3748096/7744$ &$3.54\times 10^{-4}$ & $291$ & $194$ &$6.87$\\
\hline
\hline \end{tabular}
\caption{\small Numerical solution for a problem of scattering by the
  smooth Gaussian bump example for a range of frequencies, including
  high-frequency cases. Approximately $9.4$ points per shortest
  wavelength (which occurs at ${\bf x}=0$) were used for the
  $\kappa = 50$ through $\kappa = 800$ examples (for $\kappa =800$ the
  computational domain is two-hundred five shortest wavelengths in
  size). For these experiments both the GMRES residual and the
  coupling parameter $\beta$ were set to $10^{-5}$.}
 \label{Table:scat_by_large_smooth_scat}
 \end{center}
 \end{table}

 \begin{exmp}\label{ex_fs}
 	(\textit{Fourier Smoothing  and scattering by a discontinuous refractive index distribution})
  \end{exmp}

  This example demonstrates the character of the proposed Filtered
  Fourier Smoothing strategy (Section~\ref{FS}) for penetrable
  inhomogeneous media with discontinuous refractivity---via an
  application to the canonical problem of scattering by a circular
  scatterer.  For this experiment we considered a circular scatterer
  $\mathcal{D}$ of diameter $d=1$, and with discontinuous refractive
  index given by $n^{2}(\bsx)=2$ for $\bsx \in \mathcal{D}$ and
  $n^{2}(\bsx)=1$ for $\bsx \not \in \mathcal{D}$. The computational
  domain $\Omega = (-.51,.51)\times (-.51,.51)$ was used.  An incident
  wave of the form $u^{i}(\bsx)=J_{0}(\kappa|\bsx|)$ was assumed,
  where $J_{0}$ is the Bessel function of the first kind of order
  zero. With this incident wave a closed form expression for the
  solution of the problem (\ref{HE})-(\ref{eq:-Sommerfeld}) is
  known~\cite{andersson2005fast}. Table~\ref{Table:scat_by_circularInclusion}
  presents errors obtained in the numerical solution with and without
  Fourier smoothing, and including regular Fourier
    smoothing (FS) and filtered Fourier smoothing (FFS), for various
    discretization levels, clearly demonstrating the quadratic
    convergence of the FFS-based approach, the slower and rather
    erratic convergence in absence of Fourier smoothing, and the
    improvements resulting from the use of filtering.
    Table~\ref{Table:Scat_By_Large_Disc}, in turn, concerns the
    character of the FFS method under high-frequency illumination,
  displaying fixed accuracies (of the order of three digits in this
  case), for the $n^{2}(\bsx)=2$ scatterer $\mathcal{D}$ just
  considered and for problems up to $276\cdot \lambda_\mathrm{int}$ in
  diameter, where $\lambda_\mathrm{int}= \frac{2\pi}{n\kappa}d$
  denotes the wavelength in the interior of $\mathcal{D}$. We note
  that a fixed accuracy is maintained using 11 points per wavelength,
  demonstrating, additionally, the dispersionless character of the
  algorithm even under a discontinuous index of refraction, for which
  the accuracy of the algorithm is reduced to second
  order. Table~\ref{Table:Scat_By_Large_Contrast}, finally, presents
  numerical results for highly refractive scatterers. In contrast with
  the behavior observed in the case of high-frequency illumination, in
  the present case, in which high-frequencies result from
  corresponding large refractive indexes, the iteration numbers
  required to maintain accuracy remain fixed as the refractivity
  values are increased---on account of the direct matrix solution used
  for the interior problem, and in spite of the resulting
  high-frequency interior scattering phenomenology.

\begin{table}[H]
\begin{center}
 \begin{tabular}{c|c|c|c|c|c|c|c|c}
\hline \hline
$\kappa d$ &$P\times P$ &$q\times q$& \multicolumn{2}{c|}{Without FS}  & \multicolumn{2}{c|}{With FS (without filter)}	 & \multicolumn{2}{c}{With FFS (incl. filter)} \\
\cline{4-9} & & & $\epsilon_{\infty}^{N}$& Order & $\epsilon_{\infty}^{N}$ &Order  & $\epsilon_{\infty}^{N}$ &Order \\
\hline
$10 \pi$& $ 3\times 3 $& $10 \times 10$& $8.80  \times 10^{-2}$ &-    & $1.66 \times 10^{-1}$  & -   & $2.14 \times 10^{-1}$& - \\
$10 \pi$& $6\times 6 $& $10 \times 10$& $1.27 \times 10^{-2}$ &$ 2.79 $ &  $1.60\times 10^{-2}$ & 3.37 &$4.60 \times 10^{-2}$ &$2.22$ \\
$10 \pi$&  $12\times 12 $&  $10 \times 10$ & $1.22\times 10^{-2}$ & $0.06$ &  $5.49\times 10^{-4}$ & 4.86  &$1.43\times 10^{-3}$ & $5.00$  \\
$10 \pi$&$ 24\times 24$& $10 \times 10$ & $3.30 \times 10^{-3}$ & $1.88$& $2.50\times 10^{-4}$&  1.13&$3.18 \times 10^{-4}$ &$2.17$  \\
$10 \pi$& $ 48 \times 48 $&  $10 \times 10$& $2.66 \times 10^{-3}$ & $0.31$ & $8.87\times 10^{-5}$ & 1.49& $5.09\times 10^{-5}$ & $2.64$   \\
$10 \pi$& $ 96\times 96 $& $10 \times 10$& $5.72 \times 10^{-4}$ & $2.22$   &  $4.82\times 10^{-5}$ &0.87 & $1.23\times 10^{-5}$ &$2.04$ \\
$10 \pi$&$ 192\times 192 $&  $10 \times 10$& $1.69 \times 10^{-4}$ &$1.76$&   $1.08\times 10^{-5}$ & 2.15 &$3.00 \times 10^{-6} $& $2.03$ \\
\hline
\hline
\end{tabular}
\caption{\small Demonstration of the quadratic convergence of the FS and
  FFS-based hybrid solvers for a problem of scattering by a circular
  inclusion of diameter $d=1$ with $u^{i}(\bsx)=J_{0}(\kappa |\bsx|)$
  and with discontinuous refractive index is given by $n^{2}(\bsx)=2$
  for $\bsx \in \mathcal{D}$ and $n^{2}(\bsx)=1$ for
  $\bsx \not \in \mathcal{D}$. For these experiments the GMRES
  residual tolerance and the coupling parameter $\beta$ were set to
  $10^{-6}$ and $10^{-5}$, respectively.  The beneficial effects of
  Fourier smoothing and filtering, which lead to higher accuracies and
  a more predictable convergence behavior, can be clearly
  appreciated. \label{Table:scat_by_circularInclusion}}

\end{center}
\end{table} 

\begin{table}[H]
\begin{center}
\begin{tabular}{c|c|c|c|c|c|c|c}
 \hline
 \hline
 $\kappa$& $P\times P$ &$q\times q$ & $N/ \Gamma_{N}$  &$\epsilon_{\infty}^{N}$ &$\#$ Iter.& \multicolumn{2}{c}{Time (sec.)} \\
 \cline{7-8}  & & & & & &pre-comp& per. It.\\
 \hline
 \hline
 $50 $ & $10 \times 10$ & $10 \times 10$ & $14641/484$  & $2.02  \times 10^{-3}$ & $15$ & $0.61$ &$0.03$\\
 \hline
 $100$ & $22\times 22$ & $10 \times 10$  &   $ 58564/968$& $2.55 \times 10^{-3}$ & $37$ & $3$ &$0.09$\\
 \hline
 $200$ & $44\times 44$ & $10 \times 10$  &  $245025/1980$& $2.78\times 10^{-3}$ & $69$ & $11$ &$0.42$\\
 \hline
 $400$ & $88\times 88$ & $10 \times 10$  & $980100/3960$& $3.80\times 10^{-3}$ & $161$ & $46$ &$1.72$\\
 \hline
 $800$ & $176\times 176 $ & $10 \times 10$  &  $3920400/7920$ &  $2.81\times 10^{-3}$ & $281$ & $200$ &$7.25$\\
 \hline
 $1200$ & $264\times 264$ & $10 \times 10$  &  $8433216/11616$&$3.90\times 10^{-3}$ & $400$ & $470$ &$16.3$\\

 \hline
 \hline
\end{tabular}
\caption{\small High-frequency scattering problem. Numerical solution,
  using FFS, for a problem of scattering by a circular inclusion of
  diameter $d=1$, with $u^{i}(\mathbf{x})=J_{0}(\kappa|\mathbf{x}|)$,
  and with $n^{2}(\mathbf{x})=2$ for
  $\mathbf{x} \in \mathcal{D}$ and $n^{2}(\mathbf{x})=1$
  otherwise. For these experiments the GMRES residual tolerance and
  coupling parameter $\beta$ were set to $10^{-5}$.}
\label{Table:Scat_By_Large_Disc}
%	\label{Table:scat_by_large_smooth_scat}
\end{center} 
\end{table}

\begin{table}[H]
\begin{center}
\begin{tabular}{c|c|c|c|c|c|c|c|c} \hline
\hline			
$P\times P$ &$q\times q$ & $N/\Gamma_{N}$ & ref-index & \# $\lambda_{\mathrm{int}}$ &$\epsilon_{\infty}^{N}$ & $\#$ Iter. &  \multicolumn{2}{c}{Time (sec.)} \\ 			
\cline{8-9} & & & $n(\bsx)$  & & &  &pre-comp& per. It.\\ 
\hline

$125\times 125$ & $10 \times 10$ & $1890625/5500$ &  $50 $& $125$ & $3.43 \times 10^{-3}$ & $10$ & $84$ &$3.4$  \\
\hline
$150\times 150$ & $10 \times 10$& $2722500/6600$   &$60$  & $150$ &  $4.52 \times 10^{-3}$ & $10$ & $125$ &$5.0$ \\
\hline
$200\times 200$ & $10 \times 10$&  $4840000/8800$  & $80$  & $200$ &  $4.15\times 10^{-3}$ & $12$ & $232$ &$9.3$ \\
\hline
$250\times 250$ & $10 \times 10$ & $7562500/11000$ & $100$& $250$ & $3.73\times 10^{-3}$ & $12$ & $403$ &$14.0$ \\
\hline
$350\times 350$ & $10 \times 10$ & $14822500/15400$ & $140$& $350$ & $3.34\times 10^{-3}$ & $12$ & $848$ &$29.0$ \\
\hline
\hline
\end{tabular}  
\caption{\small Large contrast scattering problem. Numerical solution,
  using FFS, for a problem of scattering by a circular inclusion of
  diameter $d=2$ refractive index $n(\bsx)$ (resp. refractive index
  $1$) in the interior (resp. the exterior) of the inclusion, with
  $u^{i}(\mathbf{x})=J_{0}(\kappa|\mathbf{x}|)$, where $\kappa=5
  \pi$. Three digit accuracy is maintained using 11 points per
  wavelength. For these experiments the GMRES residual tolerance and
  coupling parameter $\beta$ were set to $10^{-5}$.}
\label{Table:Scat_By_Large_Contrast}
%	\label{Table:scat_by_large_smooth_scat}
\end{center} 
\end{table}

Figure~\ref{fig:-circular_inclusion} provides a graphical depiction of the scattering pattern obtained for a circular
scatterer $\mathcal{D}$ of diameter $d=2$, under incident illumination given by $u^{i}(\bsx)=\exp(i\kappa x_{1})$, with $\kappa=100$ and with $n^{2}(\bsx)=3$ for $\bsx \in \mathcal{D}$ and $n^{2}(\bsx)=1$ for
$\bsx$  outside $\mathcal{D}$---for which we have $d=55 \lambda_{\mathrm{int}}$. Using 12 points per wavelength the method achieves three digits of accuracy for this problem in the near field in a two and half minutes single-core computation, including both, precomputation and all necessary iterations.

\begin{figure}[H]
\centering
\begin{minipage}[b]{0.351\linewidth}
\includegraphics[width=\linewidth]{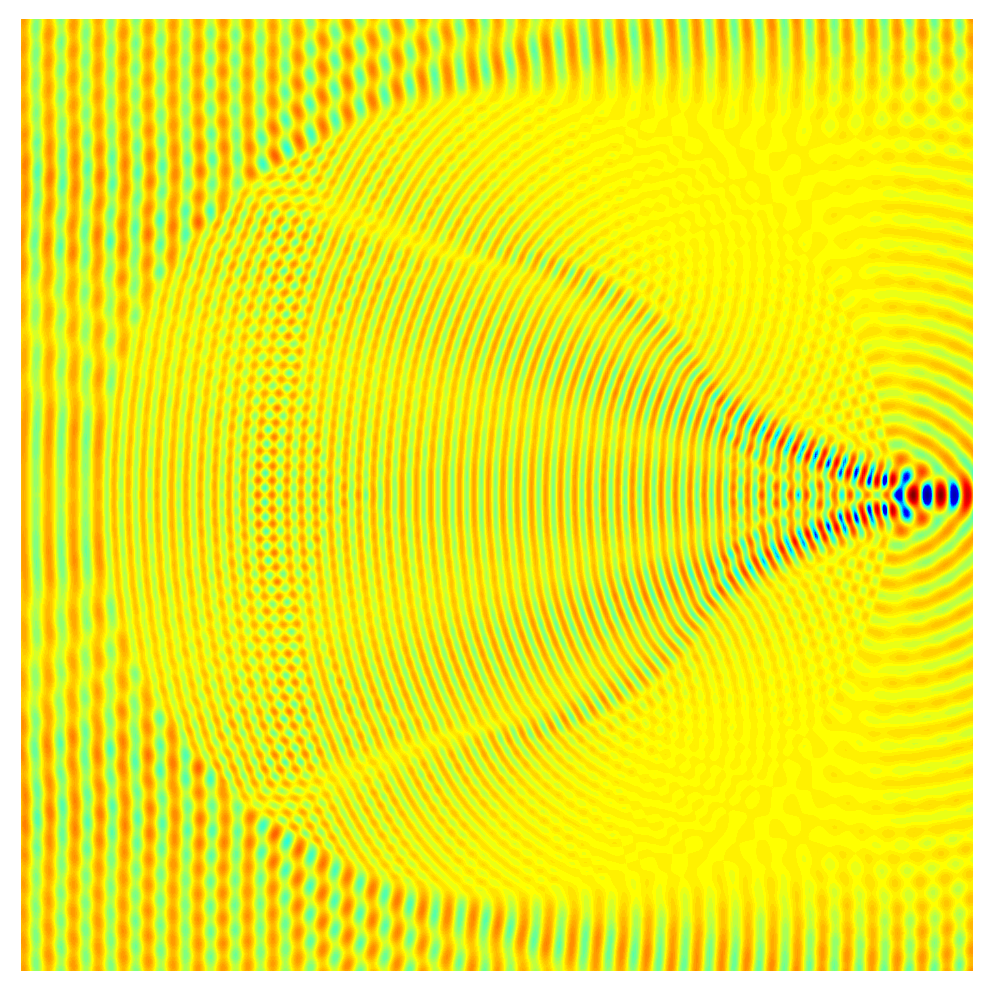}
\subcaption{Real part of the total Field $u$.}
%	\label{fig:chapter001_dist_001}
\end{minipage}
\hspace{0.5cm}
\begin{minipage}[b]{0.347\linewidth}
\includegraphics[width=\linewidth]{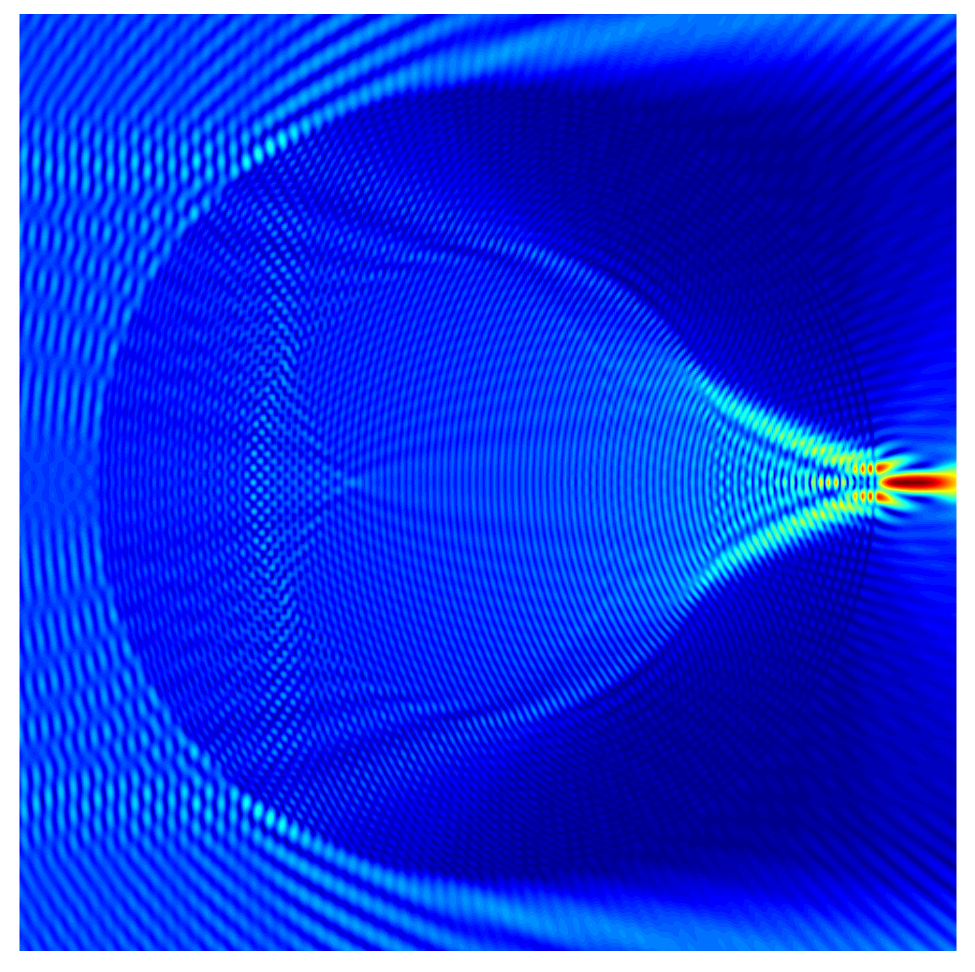}
\subcaption{Absolute value of the total field $u$.}
%	\label{fig:chapter001_reward_001}
\end{minipage}
\caption{\small Scattering of the plane wave $ \exp (i \kappa x_{1})$ with
  $\kappa=100$ by a penetrable circular inclusion $\mathcal{D}$ of
  diameter $d = 55 \lambda_{\mathrm{int}}$ with $n^{2}(\bsx)=3$ for
  $\bsx \in \mathcal{D}$ and $n^{2}(\bsx)=1$ otherwise. Using 12
  points per wavelength and relying on the FFS method, the algorithm
  produced this three-digit accurate solution in a two and half
  minutes single-core computation. }
\label{fig:-circular_inclusion}
\end{figure}

\begin{exmp}(\textit{Scattering by variable discontinuous refractivity})
\end{exmp}

Our next example demonstrates the properties of the solver, including
FFS, when applied to a scatterer containing continuously variable
material properties as well as discontinuities across a material
interface. We thus consider the problem of evaluation of the total
field $u$ that results for the refractive-index distribution
\begin{equation}\label{eq:-Gaussian}
n^{2}(\bsx)=\begin{cases}
3+2e^{-4|\bsx|^{2} } & \text{if}\hspace{1mm}\bsx\in\mathcal{D},\\
1 & \text{otherwise},
\end{cases}
\end{equation}
where $\mathcal{D}$ is circular inclusion of unit radius, under the
plane wave incidence $u^{i}(\bsx)= \exp(i\kappa x_{1})$. Since
analytical solutions are not available in this case, we use numerical
solution obtained on a finer grids for reference.  The numerical
results reported in Table~\ref{Table:scat_by_Gaussian} display errors
that in fact decrease faster than the quadratic rate expected from use
of the FFS approach.

\begin{table}[h!] 
\begin{center}
\begin{tabular}{c|c|c|c|c|c|c} 
\hline \hline
$\kappa d$ & $P\times P$ &$q\times q$&  \multicolumn{2}{c|}{Without FFS} 	 & \multicolumn{2}{c}{With FFS} \\
\cline{4-7} & & & $\epsilon_{\infty}^{N}$& Order & $\epsilon_{\infty}^{N}$ &Order \\ 
\hline
$10 \pi$& $ 3\times 3 $&$10 \times 10$& $1.15 \times 10^{0}$ &- & $1.06\times 10^{-0}$& - \\ 
$10 \pi$& $6\times 6 $&$10 \times 10$& $1.45 \times 10^{-1} $& $2.98$ & $1.18 \times 10^{-1}$ &$3.16$\\ 
$10\pi$&  $ 12\times 12 $&$10 \times 10$ &$5.48 \times 10^{-2}$& $1.40$ & $1.28\times 10^{-2}$ &$3.20$\\ 
$10 \pi$&$ 24\times 24$&$10 \times 10$  &$1.48\times 10^{-2}$ &$1.88$ & $1.53\times 10^{-3}$ & $3.06$\\ 
$10 \pi$&$48\times 48 $& $10 \times 10$ &$7.73\times 10^{-3}$& $0.93$ & $3.34\times 10^{-4}$ & $2.19$ \\ 
$10 \pi$&$ 96\times 96 $&$10 \times 10$ & $4.24\times 10^{-3}$ & $.86$ &$7.62\times 10^{-5}$ &$2.13$ \\ 
$10 \pi$& $ 192\times 192 $&$10 \times 10$ &$8.32 \times 10^{-4}$ &2.34 & $1.42\times 10^{-5}$& 2.42\\ 
\hline	
\hline
\end{tabular}  
\caption{\small Convergence Study: Convergence of the proposed algorithm for
  scattering for a discontinuous-refractivity problem as
  in~(\ref{eq:-Gaussian}).  An incident plane wave incoming from the
  positive $x$-axis was used in this case. As noted in the text, the
  errors decrease somewhat faster than the quadratic rate expected
  from use of the FFS approach.}
\label{Table:scat_by_Gaussian}
\end{center} 
\end{table}
Figure~\ref{fig:-GaussianRefractivity} displays near fields obtained
for the discontinuous refractive-index~(\ref{eq:-Gaussian}) under
wavenumbers $\kappa =50$ and $\kappa=150$, for which the diameters of
inhomogeneity are $36 \lambda_{\mathrm{min}}$ and
$108 \lambda_{\mathrm{min}}$, respectively, where
$\lambda_{\mathrm{min}}$ denotes the smallest interior wavelength.  In
both the cases, the algorithm achieved three-digit accuracy in the
near field by using 11 points per $\lambda_{\mathrm{min}}$ in
single-core computations requiring one and thirteen minutes,
respectively.

\begin{figure}[H] 
\begin{minipage}[b]{0.33\linewidth}
\centering
\includegraphics[width=\linewidth]{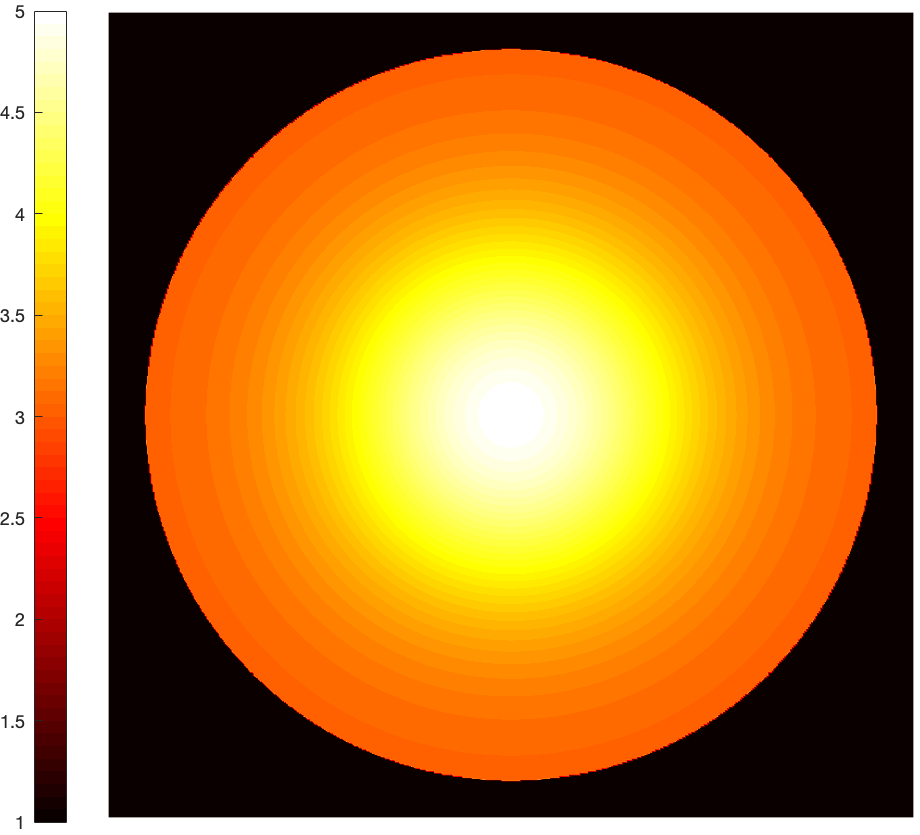}
\subcaption{Gaussian Refractivity.}
%	\label{fig:chapter001_dist_001}
\end{minipage}
\hspace{0.3cm}
\begin{minipage}[b]{0.3\linewidth}
\centering
\includegraphics[width=\linewidth]{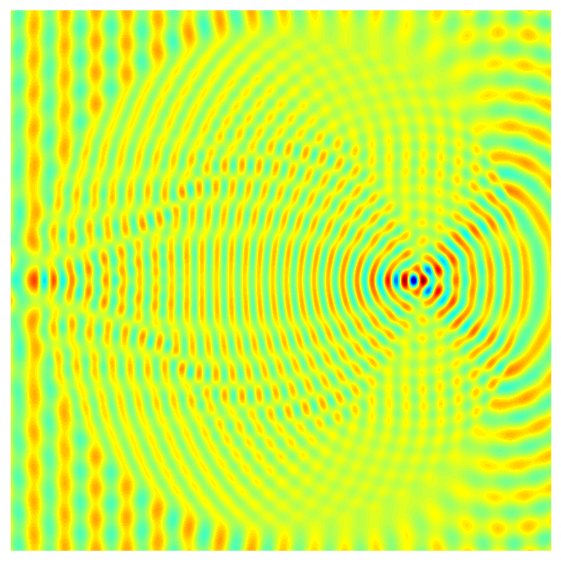}
\subcaption{$\kappa=50$, Real part of ${u}$}
%	\label{fig:chapter001_reward_001}
\end{minipage}
\hspace{0.3cm}
\begin{minipage}[b]{0.3\linewidth}
\centering
\includegraphics[width=\linewidth]{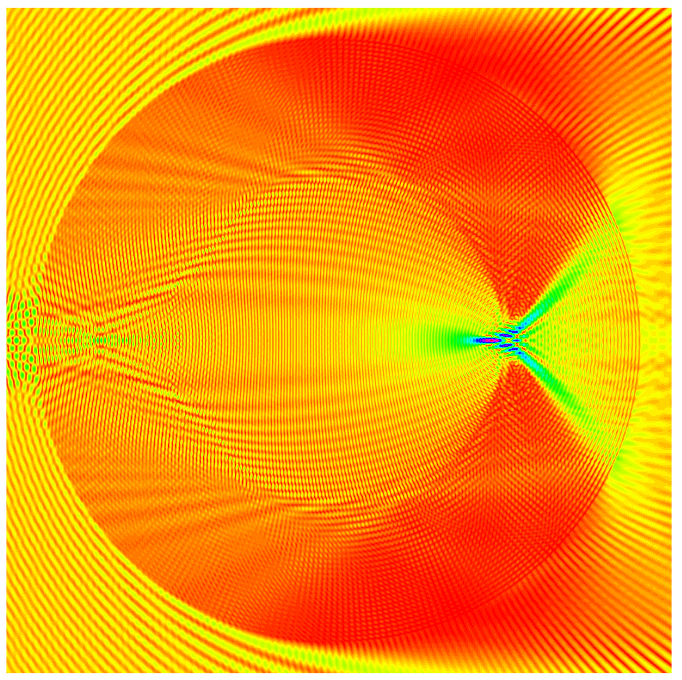}
\subcaption{$\kappa=150,$ $|u|$. }
%	\label{fig:chapter001_reward_001}
\end{minipage}

\caption{\small Scattering of a plane wave $ \exp (i \kappa x_{1})$ by a the
  Gaussian refractivity profile~(\ref{eq:-Gaussian}) for $\kappa=50$
  and $\kappa=150$. In both cases, using the FFS method and 11
  points per wavelength the algorithm produced three-digit accuracy in
  a one and thirteen-minute computation respectively. }
\label{fig:-GaussianRefractivity}
\end{figure}

\begin{exmp}(\textit{Scattering by geometries containing corners and
cusps})
\end{exmp} 
None of the algorithmic components, nor the resulting accuracies in
the proposed method, are constrained in any way by the geometry of the
scatterer.  Without any additional effort, the approach can easily
deal with arbitrarily complicated geometries. To demonstrate this, we
consider two additional geometries, containing corner- and
cusp-singularities, respectively. Once again the accuracy of any one
solution is evaluated by comparison with results obtained on finer
grids. In both cases we compute the near field solution $u$ under the
plane wave incidence $u^{i}(\bsx)= \exp(i\kappa x_{1})$.

\begin{table}[H]
\begin{minipage}[t]{0.49\linewidth}\centering
\centering
\begin{tabular}{c|c|c|c|c} \hline
\hline
$\kappa d$ & $P\times P$& $q\times q$ &$\epsilon_{\infty}^{N}$ & Order\\ 		
\hline
\hline
$12 \pi$&  $ 3\times 3$& $10 \times 10$ & $ 1.33\times 10^{0}$ & - \\ 
$12 \pi$ & $6\times 6 $& $10 \times 10$ & $6.95\times 10^{-2}$ & $4.26$\\ 
$12 \pi$ &  $ 12\times 12 $& $10 \times 10$  &$9.75\times 10^{-3}$ & $2.83$ \\ 
$12 \pi$ &   $24\times 24$& $10 \times 10$  &$ 2.12 \times 10^{-3}$ &$2.20$\\ 
$12 \pi$ &  $ 48\times 48 $&$10 \times 10$  &$4.66\times 10^{-4}$ &$2.18$ \\ 
$12 \pi$ &  $ 96\times 96 $&$10 \times 10$  &$9.22\times 10^{-5}$ &$2.34$ \\ 
$12 \pi$ &  $ 192\times 192 $&$10 \times 10$ &$1.92\times 10^{-5}$ &$2.26$ \\ 
\hline	
\hline
\end{tabular}  
\caption{\small Convergence Study: Illustration of quadratic convergence of the proposed algorithm for a geometry containing a corner singularity.}
\label{Table:scat_by_Square}
\end{minipage}\hfill
\begin{minipage}[t]{0.49\linewidth}\centering
\centering
\begin{tabular}{c|c|c|c|c} \hline
\hline
$\kappa d$ & $P\times P$& $q\times q$  &$\epsilon_{\infty}^{N}$ & Order\\ 		
\hline
\hline
$12 \pi$&  $ 3\times 3$& $10 \times 10$ & $ 1.42\times 10^{0}$ & - \\ 
$12 \pi$ & $6\times 6 $& $10 \times 10$ & $9.86\times 10^{-2}$ & $3.84$\\ 
$12 \pi$ &  $ 12\times 12 $& $10 \times 10$    &$1.21\times 10^{-2}$ & $3.02$ \\ 
$12 \pi$ &   $24\times 24$& $10 \times 10$  &$ 2.26 \times 10^{-3}$ &$2.42$\\ 
$12 \pi$ &  $ 48\times 48 $&$10 \times 10$   &$6.23\times 10^{-4}$ &$1.85$ \\ 
$12 \pi$ &  $ 96\times 96 $&$10 \times 10$   &$1.46\times 10^{-4}$ &$2.09$ \\ 
$12 \pi$ &  $ 192\times 192 $&$10 \times 10$  &$2.15\times 10^{-5}$ &$2.76$ \\ 
\hline	
\hline
\end{tabular}  
\caption{\small Convergence Study: Illustration of quadratic convergence of the proposed algorithm for a geometry containing a cusp singularity.}
\label{Table:scat_by_Cusp}
\end{minipage}\hfill
\end{table}

Table~\ref{Table:scat_by_Square} presents numerical results for the scatterer $\mathcal{D}$ depicted in
Figure~\ref{fig:-scat_by_square}(a), with $n^{2}(\bsx)=2$ for $\bsx\in\mathcal{D}$ and one otherwise.  The computed near field for $\kappa=200$, which was determined to be accurate up to three digits, is
displayed in Figure~\ref{fig:-scat_by_square}(b). Table~\ref{Table:scat_by_Cusp}, in turn, presents numerical results for the scatterer $\mathcal{D}$ depicted in Figure \ref{fig:-scat_by_star}(a), which equals the region
contained between the four unit discs centered at $(1,1)$, $(1,-1)$, $(-1,1)$ and $(-1,-1)$, with $\kappa d=12\pi$, and with $n^{2}(\bsx)=2$ for $\bsx \in \mathcal{D}$ and $n^{2}(\bsx)=1$ otherwise.  Figure \ref{fig:-scat_by_star}(b), finally, displays the near field for this geometry, but with $\kappa=20\pi$ and $n^{2}(\bsx)=16$ for $\bsx \in \mathcal{D}$ and one otherwise---thus yielding a scatterer $80\lambda_\mathrm{int}$ in size. A two-digit
solution was obtained using merely nine points per wavelength and computing time of seven minutes.

\begin{figure}[h!]
\begin{center}
\begin{minipage}[b]{0.445\linewidth}
\includegraphics[width=\linewidth]{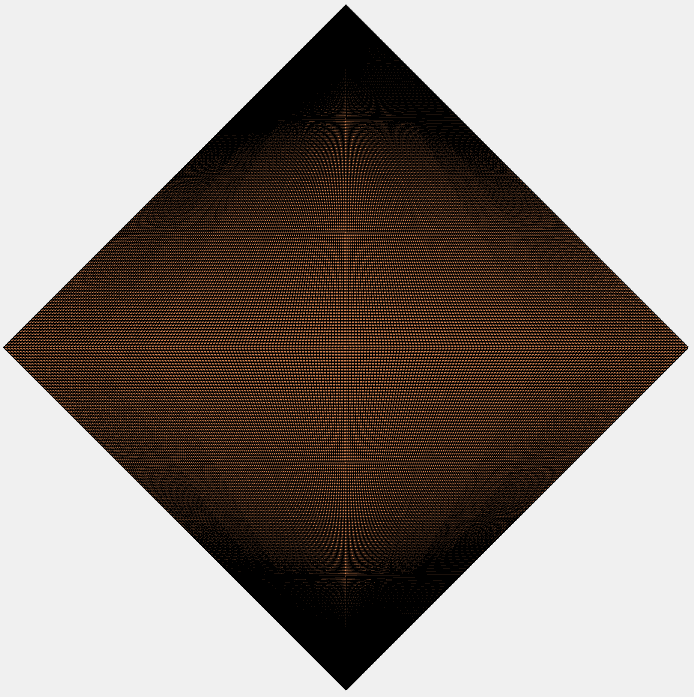}
\subcaption{Square scatterer}
%	\label{fig:chapter001_dist_001}
\end{minipage}
\hspace{0.6cm}
\begin{minipage}[b]{0.45\linewidth}
\includegraphics[width=\linewidth]{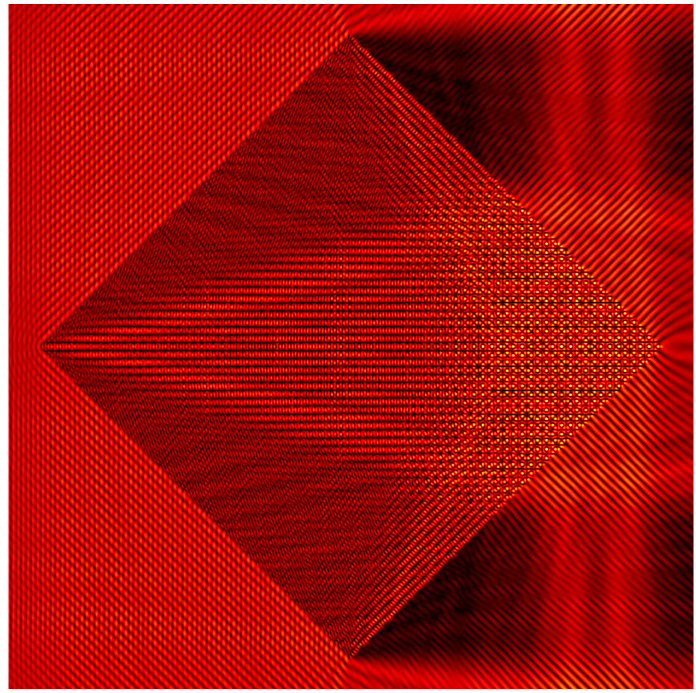}
\subcaption{Absolute value of the total field $u$.}
%	\label{fig:chapter001_reward_001}
\end{minipage}
\end{center}
\caption{\small Scattering by a geometry containing corner singularities, with $n^{2}(\bsx)=3$ for $\bsx \in \mathcal{D}$ and one otherwise. For this experiment the incident field $u^{i}(\bsx)= \exp (i \kappa x_{1})$ with $\kappa =200$ was
used. Errors of the order of $10^{-3}$ were obtained in the near field solution and the total computing time is fourteen
minutes.}
\label{fig:-scat_by_square}
\end{figure}

\begin{figure}[h!]
  \begin{center}
    \begin{minipage}[b]{0.445\linewidth}
\includegraphics[width=\linewidth]{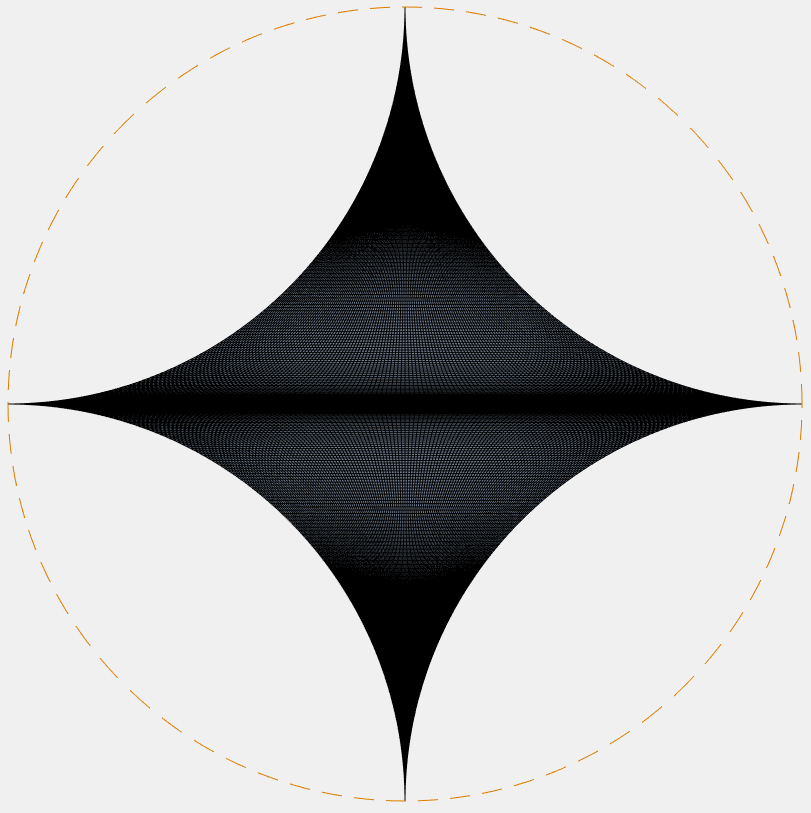}
\subcaption{Star-shaped geometry with cusp}
%	\label{fig:chapter001_dist_001}
\end{minipage}
\hspace{0.5cm}
\begin{minipage}[b]{0.46\linewidth}
\includegraphics[width=\linewidth]{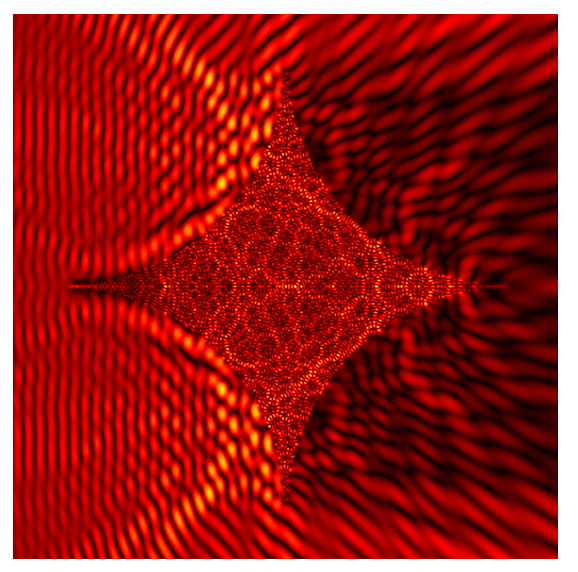}
\subcaption{Absolute value of total field $u$}
%	\label{fig:chapter001_reward_001}
\end{minipage}
\end{center}
\caption{\small Scattering of the incident field
$u^{i}(\bsx)= \exp (i \kappa x_{1})$, with $\kappa =20\pi$, by a geometry containing cusp singularities, with
$n^{2}(\bsx)=16$ for $\bsx \in \mathcal{D}$ and $n^{2}(\bsx)=1$ otherwise. Errors of the order of $10^{-2}$
were obtained in the near field solution on the basis of nine points per interior wavelength. }
\label{fig:-scat_by_star}
\end{figure}

\section{Conclusions}
\label{conclusion}
This paper introduced a new methodology for solutions of
two-dimensional problems of scattering by penetrable inhomogeneous
media with possibly discontinuous refractivity. The solver achieves
high-order convergence for smooth refractivities at nearly-linear
computing cost, and, to the best the of our knowledge, it is the first
hybrid direct/iterative solver which yields second
order convergence for discontinuous refractivities, and for low- or
high-frequencies alike.  The method additionally enjoys very low
dispersion for either smooth or discontinuous refractive indexes, and
it can natively and easily handle scatterers with complicated
geometric singularities, including e.g. as corners and cusps.
Extensions of the proposed approach to electromagnetic and elastic
wave scattering problems, as well as three-dimensional configurations
are envisioned.

\section*{Acknowledgments} 
This work was supported by NSF, DARPA and AFOSR through contracts
DMS-2109831 and HR00111720035, and FA9550-21-1-0373, and by the NSSEFF
Vannevar Bush Fellowship under contract number N00014-16-1-2808.

\appendix
\section{Appendix: Fast and accurate computation of Fourier coefficients of discontinuous functions\label{append}} 
In order to enable fast and accurate evaluation of the Fourier coefficients of a given, possibly discontinuous, function $f$ in the interval $[0,2\pi]$, as needed in Section~\ref{FS} (see Remark~\ref{FC-disc}), we rely on the Fourier
continuation (FC) approach~\cite{bruno-lyon2010,amlani2016fc}. For our description we assume that the function $f$ has only one discontinuity, say, at $x=a\in(0,2\pi)$, but an arbitrary number of discontinuities may be treated in similar fashion. 

Let now $f^{c}_{j}$ ($j=1,2$) denote $d_{j}$-periodic Fourier continuation functions of the restrictions of the function $f$ to the intervals $[0,a]$ and $[a,2\pi]$, respectively.  We thus have 
\begin{equation}\label{FC2}
f^{c}_{j}(x)=\sum_{k=-F}^{F} c_{k}^{j} e^{\frac{ 2\pi i kx}{d_{j}} },
\end{equation}
where, following e.g.~\cite{amlani2016fc}, the Fourier coefficients $ c_{k}^{j}$ are obtained in $O(F \log F)$ operations by means of the FC procedure and associated FFTs, and where the resulting functions
$f^{c}_{j}$ with $j=1,2$ approximate the restrictions of the function $f$ to the intervals $[0,a]$ and $[a,2\pi]$, respectively, with high-order accuracy. Let 
\begin{equation}\label{fc_1}
f_{\ell} =\frac{1}{2\pi}\int_{0}^{2\pi}f(t)e^{- i \ell t} dt\\ 
=\frac{1}{2\pi}\int_{0}^{a}f(t) e^{- i \ell t} dt+\frac{1}{2\pi}\int_{a}^{2\pi}f(t)e^{-i \ell t} dt
\end{equation}
denote the desired Fourier coefficient of $f$ in the interval $[0,2\pi]$. The two integrals on the right-hand side of~\eqref{fc_1} can be computed with high accuracy by substituting $f$ by $f^{c}_{j}$ and exchanging integration and summation. In the case of the first integral, for example, we have
\begin{equation}\label{fc_3}
\int_{0}^{a}f(t) e^{- i \ell t} dt \approx\int_{0}^{a}f^{c}_{1}(t) e^{-i \ell t} dt\\
=\sum_{k=-F}^{F} c^{1}_{k}\int_{0}^{a}e^{\frac{2\pi k-\ell d_{1}}{d_{1}} i t} dt\\ 
=\sum_{k=-F}^{F} c^{1}_{k}b_{2\pi k-\ell d_{1}},
\end{equation}
where
\[
b_{2\pi k-\ell d_{1}} =
\begin{cases}
\frac{d_{1}}{i (2\pi k- \ell d_{1})}\left(e^{\frac{2\pi k-\ell d_1}{d_{1}} i a}-1 \right) & \mbox{if} \ (2\pi k-\ell d_1)\neq 0,\\
a& \mbox{otherwise}.
\end{cases}
\]
The summation in~\eqref{fc_3} is a discrete scaled convolution and can be obtained for all $\ell$ in $O(F\log F)$ operations by using FFT~\cite{nascov2009fast}. Thus highly-accurate values of the Fourier coefficients $f_{\ell}$ of the discontinuous function $f$, for all $\ell$, $-F\leq \ell \leq F$, can be produced in $O(F \log F)$ operations.

\bibliographystyle{abbrv} 
\bibliography{HybridSolver_Reference}

\end{document}